\documentclass[review,onefignum,onetabnum]{siamart171218}

\usepackage{lipsum}
\usepackage{amsfonts}
\usepackage{graphicx}
\usepackage{epstopdf}
\usepackage{algorithmic}
\ifpdf
  \DeclareGraphicsExtensions{.eps,.pdf,.png,.jpg}
\else
  \DeclareGraphicsExtensions{.eps}
\fi

\newsiamremark{assumption}{Assumption}
\newsiamremark{remark}{Remark}
\newsiamremark{hypothesis}{Hypothesis}
\crefname{hypothesis}{Hypothesis}{Hypotheses}
\newsiamthm{claim}{Claim}

\headers{A three-operator splitting for nonconvex sparsity regularization}{Fengmiao Bian and Xiaoqun Zhang}

\title{A three-operator splitting  algorithm for nonconvex sparsity regularization\thanks{
{This work was supported by  NSFC (No.11771288, 91630311) and  National key research and development program (No.2017YFB0202902). We thank the Student Innovation Center at Shanghai Jiao Tong University for providing us the computing services.}}}

\author{Fengmiao Bian\thanks{School of Mathematical Sciences, Shanghai Jiao Tong University, Shanghai 200240, CHINA
  (\email{bianfm17@sjtu.edu.cn}).}
\and Xiaoqun Zhang\thanks{School of Mathematical Sciences and Institute of Natural Sciences, Shanghai Jiao Tong University, Shanghai 200240, CHINA
  (\email{xqzhang@sjtu.edu.cn}).}}

\usepackage{amsopn}

\ifpdf
\hypersetup{
  pdftitle={A three-operator splitting  algorithm for nonconvex sparsity regularization},
  pdfauthor={Fengmiao Bian and Xiaoqun Zhang} 
}
\fi

\begin{document}

\maketitle

\begin{abstract}
Sparsity regularization has been largely applied in many fields, such as signal and image processing and machine learning. In this paper, we mainly consider nonconvex minimization problems involving three terms, for the applications such as: sparse signal recovery and low rank matrix recovery. We employ a three-operator splitting proposed by Davis and Yin \cite{DY} (called DYS) to  solve the resulting possibly nonconvex problems and develop the convergence theory for this three-operator splitting algorithm in the nonconvex case. We show that if the step size is chosen less than a computable threshold, then the whole sequence converges to a stationary point. By defining a new decreasing energy function associated with the DYS method, we establish the global convergence of the whole sequence and a local convergence rate under an additional assumption that this energy function is a Kurdyka-\L ojasiewicz function. We also provide sufficient conditions for the boundedness of the generated sequence. Finally, some numerical experiments are conducted to compare the DYS algorithm with some classical efficient algorithms for sparse signal recovery and low rank matrix completion. The numerical results indicate that DYS method outperforms the exsiting methods for these specific applications.
 \end{abstract}



\section{Introduction}
Sparsity regularization has been largely applied in many fields , such as signal and image processing and machine learning. In this paper, we mainly consider two applications that involving minimization of three terms with possibly nonconvex functions. For example, in \cite{ELX} Esser, Lou, and Xin first proposed using the difference of $l_1$ and $l_2$ norms as sparse regularization term. Later,  the authors in \cite{LYHX, YLHX} applied this $l_{1-2}$ metric to solve the sparse recovery problem in signal processing. In fact, the $l_{1-2}$ metric has indicated its advantages in other kinds of applications, for instance, image restoration \cite{LZOX}, phase retrieval \cite{YX}, and the anisotropic and isotropic forms of total variation discretizations \cite{YZY}. For low rank matrix recovery problems, in \cite{JMD}, Jain, Meka and Dhillon proposed a simple and fast algorithm for rank minimization under affine constraints. In \cite{CCS}, Cai, Cand\`{e}s and Shen introduced a novel algorithm to approximate the matrix with minimum nuclear norm among all matrices obeying a set of convex constraints. In \cite{CTCB}, Cabral, Torre, Costeira and Bernardino proposed a unified model to nuclear norm regularization and bilinear factorization for low-rank matrix decomposition and analyzed the conditions under which these approaches are equivalent. In general, these models can be formulated as the following type of nonconvex minimization problem:  
\begin{equation}\label{model}
\min_{x}    F(x) + G(x) + H(x),
\end{equation}
for example, for sparse recovery problems in \cite{LYHX, YLHX}, the corresponding function $F$ can be the data term, $G$ is the $l_1$ norm and $H$ is the negative $l_2$ norm; for low rank matrix decomposition in \cite{CTCB}, the corresponding function $F$ is the loss function, the functions $G$ and $H$ are regularizations of the each factorization. 

In the last decade, several optimization algorithms have been designed to work out model \eqref{model} in the noncovex setting. Most of them focus on splitting algorithms and most are based on two famous algorithms: the alternating direction method of multipliers (ADMM) and the Douglas-Rachford splitting (DRS). For the ADMM algorithm, several articles have been devoted to solving problems with a similar structure to \eqref{model}. In \cite{YPC}, Yang, Pong and Chen proved the global convergence of ADMM algorithm under the conditions that one of the summands is convex, the other is possibly nonconvex and nonsmooth, and the third is the Fr\"obenius norm. In \cite{WYZ}, Wang, Yin and Zeng considered a general nonconvex optimization problem with coupled linear equality constraints. By assuming that the objective function is continuous and coercive over the feasible set, while its nonsmooth part is either restricted prox-regular or piecewise linear, and then the authors analyzed the convergence of the sequence. In \cite{LSG} similar techniques are also used in the convergence results for a nonconvex linearized ADMM algorithm. Bolte, Sabach, and Teboulle formulated in \cite{BST}, also in the nonconvex setting, a proximal alternating linearization method (PALM) for solving minimizing objective functions consisting of three summands: two nonsmooth functions and a smooth function which couples the two block variables.  In \cite{BCN}, Bot, Csetnek and Nguyen proposed a proximal ADMM algorithm for a class of similar objective functions consisting of three summands, but one of which is the  composition of a nonsmooth function with a linear operator, and  they proved that any cluster point of the sequence is a KKT point of the minimization problem. We can see that from above, there have been many theoretical analyses on ADMM algorithm in the nonconvex and nonsmooth setting. However, there are only a few works on DRS method for nonconvex nonsmooth optimization problem. For the model \eqref{model} when $H=0$, Li and Pong in \cite{LP} applied the DRS algorithm to the nonconvex feasibility problems and established the convergence of the DRS method when $F$ has a Lipschitz continuous gradient and $G$ is a possibly nonconvex nonsmooth function. In \cite{GHY}, when $F$ is strongly convex, $G$ is weakly convex, and $F + G$ is strongly convex, Guo, Han and Yuan showed that the sequence generated by DRS method is Fej\`{e}r monotone with respect to the set of fixed points of DRS operator, thus convergent. In \cite{LLP}, Li, Liu and Pong showed that a variant of DRS method, i.e., Peaceman-Rachford splitting, is convergent under the assumptions that $F$ is  a strongly convex Lipschitz differentiable function and $G$ is a nonconvex nonsmooth function. In \cite{TP}, Themelis and Patrinos employed the Douglas-Rachford envelope to unify and simplify the global convergence theory for ADMM, DRS and PRS in nonconvex setting. In \cite{BZ}, we generalized the DRS algorithm and proved its global convergence under the similar conditions in \cite{LP}. It is shown that this parameterized DRS algorithm perform well for some applications in data sicence.

%
%

Recently, Davis and Yin \cite{DY} proposed a new three-operator splitting method, called Davis-Yin splitting (DYS), for solving inclusion problems with three maximal monotone operators by designing a nicely behaved fixed-point equation, which extends the Douglas-Rachford and forward-backward equations. Since the subdifferentials of nonconvex functions are generally non-monotone, the existing results in \cite{DY} apply only to model \eqref{model} when $F$, $G$ and $H$ are all convex functions. In this paper, we intend to apply DYS method to resolve nonconvex problems with three terms arising in sparsity regularization. The minimization of the objective function is decomposed into solving two individuals proximal mapping. In addition, for many sparsity regularization, such as indictor function and $\| \cdot \|_0$, their proximal mappings have explicit solutions. For applying DYS method to these nonconvex problems, we first need to establish the corresponding convergence theory. For model \eqref{model} where all the three functions are possibly nonconvex, Liu and Yin \cite{LY} introduced an envelope function for DYS and showed that the global minimizers, local minimizers, critical (stationary) points, and strict saddle points of the envelope function correspond one on one to those of the objective function in model \eqref{model} under smoothness conditions for $F$ and $H$. However, there are no available convergence results for DYS in the non-convex case in paper \cite{LY}. In this paper, we will construct a new energy function to study the convergence of Davis-Yin splitting in the nonconvex setting. 

\begin{algorithm}[htp]
\caption{ Davis-Yin Splitting Algorithm}
\label{alg:DYS}
\begin{algorithmic}
\STATE{{\bf{Step 0.}}  Choose a step-size $\gamma >0$ and an initial point $x^{0}.$}

\STATE{{\bf{Step 1.}}  Set
\begin{subequations}\label{algalg}
\begin{align}
&y^{t+1} \in  \arg\min_{y}  \Bigg\{  F(y) + \frac{1}{2\gamma}\|y-x^{t}\|^{2} \Bigg\}, \label{algy}\\
&z^{t+1} \in  \arg\min_{z} \Bigg\{ G(z) + \frac{1}{2\gamma}\|z - ( 2y^{t+1} - \gamma \nabla H(y^{t+1}) - x^{t} ) \|^{2} \Bigg\},  \label{algz} \\
&x^{t+1} = x^{t} + (z^{t+1} - y^{t+1}). \label{algx}
\end{align}
\end{subequations}
}
\STATE{{\bf{Step 2.}}  If a termination criterion is not met, go to Step 1.}

\end{algorithmic}
\end{algorithm}

We present the form of DYS in the nonconvex case in \cref{alg:DYS}. In general, subproblems \eqref{algy} and \eqref{algz} are simpler to solve, so DYS method decompose a difficult optimization problem into simpler subproblems. At the same time, we can see two special cases from the DYS algorithm:

 (i) When the function $H(x)$ in model \eqref{model} is equal to 0, the DYS algorithm becomes the classical DRS algorithm; 
 
 (ii) When the function $F(x)$ in the model \eqref{model} is equal to 0, the DYS algorithm becomes a another very popular algorithm, namely forward-backward splitting algorithm (FBS). 
 
 Therefore, the DYS algorithm is an extension of these two classical algorithms, but as as mentioned before, the DYS algorithm still lacks certain theoretical analysis and applications in the nonconvex setting. This will be the main content of this paper. In summary, the contributions to this article are as follows:

1. We show that when the step size in the algorithm \eqref{alg:DYS} is less than some computable threshold, any cluster point of the sequence generated by algorithm \eqref{alg:DYS} is a stationary point of model \eqref{model}. We achieve this by revealing that the sequence is decreasing along a new energy function associated with the DYS method.

2. We establish the global convergence of the sequence generated by DYS method when the energy function meets Kurdyka-\L ojasiewicz (KL) conditions. We also give some sufficient conditions to guarantee the boundedness of the sequence generated by DYS method. Furthermore, we prove a local convergence rate of DYS method when the energy function is a KL function.

3. We resolve sparse signal recovery and low rank matrix recovery problems by DYS method, and the experiments results indicate that DYS method outperforms the exsiting methods for these specific applications. Especially for the low rank matrix recovery problem, DYS method clearly shows its advantages on the computation speed and the accuracy of the solution compared to some classical methods.

The rest of this paper is organized as follows. We present some notation and preliminaries in \cref{sec:notation}. We study the convergence behavior of DYS algorithm for a class of nonconvex and nonsmooth model \eqref{model} in \cref{sec:convergence}. In \cref{sec:examples}, we carry out some experiments with DYS algorithm, and the numerical results show that this algorithm is very efficient. In \cref{sec:conclude}, we give some concluding remarks.
\section{Notation and preliminaries}
\label{sec:notation}
In this paper, we use $\mathbb{R}^{n}$ to denote the $n$-dimensional Euclidean space, $ \left \langle \cdot, \cdot \right \rangle$ to denote the inner product and $\| \cdot \| = \sqrt{\left \langle \cdot, \cdot \right \rangle}$ to denote the norm induced by the inner product. For an extended-real-valued function $f : \mathbb{R}^{n} \to (-\infty, \infty]$, $f$ is said to be proper if it is never $-\infty$ and its domain, dom $f := \{x \in \mathbb{R}^{n} : f(x) < +\infty\}$ is nonempty. The function is called closed if it is proper and lower semicontinuous. 

 For a  proper function $f$, the limiting {\it subdifferential} of $f$ at $x \in $ dom $f$ is defined by
\begin{equation}\label{subdiff}
\aligned
\partial f(x) :=  \Big\{  &  v \in \mathbb{R}^{n} : \exists x^{t} \to x, f(x^{t}) \to f(x), v^{t} \to v ~~ \textmd{with} \\
& \liminf_{z \to x^t} \frac{f(z) - f(x^t) -  \left \langle v^t, z - x^t \right \rangle}{\| z - x^t \|} \geq 0 ~ \textmd{for} ~ \textmd{each} ~ t       \Big\}. 
\endaligned
\end{equation}  
From the above definition, we can clearly see that if $f$ is differentiable at $x$, then we have $\partial f(x) = \{ \nabla f(x) \}$. If $f$ is convex, then we have 
\begin{equation}\label{condiff}
\partial f(x) = \Big\{ v \in \mathbb{R}^{n} : f(z) \geq f(x) + \left \langle v, z-x \right \rangle~\textmd{for}~\textmd{any}~ z \in \mathbb{R}^{n} \Big\},
\end{equation}
which is the classical definition of subdifferential in convex analysis. Moreover, the inclusion property in the following
\begin{equation}\label{subdiffproperty}
\Big\{  v \in \mathbb{R}^{n} :  \exists x^{t} \to x, f(x^{t}) \to f(x), v^{t} \to v, v^{t} \in \partial f(x^{t})\Big\}\subseteq \partial f(x)
\end{equation}
holds for each $x \in \mathbb{R}^{n}$. A point $x^{*}$ is a stationary point of a function $f$ if $0 \in \partial f(x^{*})$. $x^{*}$ is a critical point of $f$ if $f$ is differentiable at $x^{*}$ and $\nabla f(x^{*}) = 0$. A function is called to be coercive if $\liminf_{\| x \| \to \infty} f(x) = \infty$. We say that $f$ is a strongly convex function with modulus $\sigma > 0$ if $f - \frac{\sigma}{2} \| \cdot \|^{2}$ is a convex function.

For any $\gamma >0$, the proximal mapping of $f$ is defined by 
\begin{equation}\label{def-pro}
P_{\gamma f} (x): x \rightarrow {\arg\min}_{y\in\mathbb{R}^n} \Big\{ f(y) + \frac{1}{2\gamma} \| y - x\|^2 \Big\}, 
\end{equation}
assuming that the $\arg \min$ exists,  where $\rightarrow$  means a possibly set-valued mapping. And for a closed set $S \subseteq \mathbb{R}^{n}$, its indicator function $\delta_{S}$ is defined by
\begin{equation}\label{def-indi}
\delta_S(x)= 
\begin{cases}
 0, ~~~~ & \textmd{if} ~ x\in S,\\
+\infty, ~~~~ & \textmd{if} ~ x\notin S. 
\end{cases}
\end{equation} 
Next, we recall some definitions related to KL function which plays an essential role in our global convergence analysis. 
\begin{definition}\label{Sem-201} $($real semialgebraic set$)$. 
A semi-algebraic set $S \subseteq \mathbb{R}^{n}$ is a finite union of sets of the form
\begin{equation}\label{semialg}
\Big\{ x \in \mathbb{R}^{n} : h_{1} (x) = \cdots h_{k} (x) = 0, ~g_{1} (x) < 0, \dots , g_{l} (x) < 0 \Big\},
\end{equation}
where $g_{1}, \dots, g_{l}$ and $h_{1}, \dots, h_{k}$ are real polynomials. 
\end{definition} 
\begin{definition}\label{Sem-202} $($real semialgebraic function$)$. 
A function $f : \mathbb{R}^{n} \to \mathbb{R}$ is semi-algebraic if the set $\big\{ (x,~f(x)) \in \mathbb{R}^{n+1} : x \in \mathbb{R}^{n} \big\}$ is semi-algebraic.
\end{definition} 

Remark that the semi-algebraic sets and semi-algebraic functions can be easily identified and contain a large number of possibly nonconvex functions arising in applications, such as see \cite{ABRS,ABS,BDL}. We also need the following KL property which holds in particular for semi-algebraic functions. 
\begin{definition}\label{KL}$($KL property and KL function$)$.
The function $F: \mathbb{R}^{n} \to \mathbb{R} \cup \{\infty\}$ has the Kurdyka-\L ojasiewicz property at $x^{*} \in$ dom $\partial F$ if there exist $\eta \in (0, \infty]$, a
neighborhood $U$ of $x^{*}$, and a continuous concave function $\varphi : [0, \eta) \to \mathbb{R}_{+}$ such that:

\begin{itemize}
\item [(i)]  $\varphi(0) = 0,~\varphi \in C^{1}((0, \eta))$, and  $\varphi^{'}(s) > 0$ for all $s \in (0, \eta)$;
\item[(ii)] for all $x \in U \cap [F(x^{*}) < F < F(x^{*})+ \eta]$ the Kurdyka-\L ojasiewicz inequality holds, i.e.,
$$ 
\varphi^{'}(F(x) - F(x^{*}))dist(0, \partial F(x)) \geq 1.
$$
\end{itemize}
If the function $F$ satisfies the Kurdyka-\L ojasiewicz property  at each point of dom $\partial F$, it is called a KL function.
\end{definition}
\begin{remark}
It follows from \cite{ABRS} that a proper closed semi-algebraic function always satisfies the KL property.
\end{remark}

\section{Convergence analysis}
\label{sec:convergence}

In this section, we analyze the convergence when \cref{alg:DYS} is applied to model \eqref{model}. For convenience, we give the corresponding first-order optimality conditions for the subproblems in \cref{alg:DYS} as follows, which will be used frequently in the convergence analysis.
\begin{subequations}
\begin{align}
& 0 \in \nabla F(y^{t+1}) + \frac{1}{\gamma} (  y^{t+1} - x^{t}), \label{opti-1}\\
& 0 \in \partial G(z^{t+1}) + \frac{1}{\gamma} (z^{t+1} + \gamma \nabla H(y^{t+1})  - 2y^{t+1} + x^{t}). \label{opti-2}
\end{align}
\end{subequations}

We will analyze \cref{alg:DYS} under the following assumptions.
\begin{assumption}\label{ass1} Functions $F,~G~and~H$ satisfy
\begin{itemize}
\item [(a1)] The function $F$ has a Lipschitz continuous gradient, i.e, there exists a constant $L  > 0$ such that
\begin{equation}\label{flip}
\| \nabla F(y_1) - \nabla F(y_2) \| \leq L \| y_1 - y_2 \|, ~~~\forall  y_1, y_2 \in \mathbb{R}^{n};
\end{equation}
\item[(a2)] $G$ is a proper closed function with a nonempty mapping $P_{\gamma G}(x)$ for any $x$ and for $\gamma > 0$;
\item[(a3)]The function $H$ has a Lipschitz continuous gradient, i.e, there exists a constant $\beta > 0$ such that
\begin{equation}\label{hlip}
\| \nabla H(y_1) - \nabla H(y_2) \| \leq \beta \| y_1 - y_2 \|, ~~~\forall y_1, y_2 \in \mathbb{R}^{n}.
\end{equation}
\end{itemize}
\end{assumption}

\begin{remark}\label{rem-ass}About the above assumptions, we can notice that
\begin{itemize}
\item[1.] When the function $H = 0$, \cref{alg:DYS} is the classical DRS algorithm. So far as we known, for DRS method in the nonconvex setting, the smoothness assumption about function $F$ has been essential. Therefore, in this case, our assumption is the same as ones in \cite{LP}. We also note that in \cite{LPADMM}, similar smoothness assumption on $F$ is also required  for the convergence analysis of ADMM algorithm in nonconvex case.
\item[2.] When the function $F = 0$, this algorithm becomes the classical Forward-Backward splitting algorithm. From \cref{alg:DYS}, we can see that the smoothness assumption about $H$ is indispensable.
\item[3.] If $F$ has a Lipschitz continuous gradient, then we can always find $l \in \mathbb{R}$ such that $F + \frac{l}{2}\| \cdot \|^{2}$ is convex, in particular, $l$ can be taken to be $L$.
\end{itemize} 
\end{remark}

Next, we start to establish the convergence, which will make use of the following energy function associated with  \cref{alg:DYS}:
\begin{equation}\label{meritfun2}
\begin{aligned}
&\Theta_{\gamma} (x, y, z) = F(y) + G(z) +H(y) + \frac{1}{2\gamma}\| 2y - z - x - \gamma \nabla H(y) \|^2 \\
&~~~~~~~~~~~~~~~~~~ - \frac{1}{2\gamma}\| x - y + \gamma \nabla H(y) \|^2 - \frac{1}{\gamma}\| y - z \|^2. 
\end{aligned}
\end{equation}
Remark that, the energy function $\Theta_{\gamma}$ is exactly the objective function  that needs to be minimized when $y=z$, which will be proved  for the limit of the sequence $\{(x^t, y^t, z^t) \}$ in \cref{clus-point}.

The following lemma states that the energy function $\Theta_{\gamma}$ decreases along the sequence generated by \cref{alg:DYS} when the step size is less than a computable threshold.
\begin{lemma}\label{decrease}
Suppose functions $F(x)$, $G(x)$ and $H(x)$ satisfy \cref{ass1}. Let $\{ (x^{t}, y^{t}, $ $z^{t}) \}$ be a sequence generated by \cref{alg:DYS}. Then for all $t \geq 1$, we have
\begin{equation}\label{de-ineq}
\Theta_{\gamma}(x^{t+1}, y^{t+1}, z^{t+1}) - \Theta_{\gamma}(x^{t}, y^{t}, z^{t}) \leq -\Lambda(\gamma) \| y^{t+1} - y^{t} \|^{2},
\end{equation}
where 
\begin{equation}\label{cond-lam}
\Lambda(\gamma) := \frac{1}{2} \left( \frac{1}{\gamma} - l \right) - \beta - (\frac{1}{\gamma} + \frac{\beta}{2}) [( -1 + 2\gamma l ) + (1 + \gamma L)^2 ].
\end{equation}
Furthermore, if the parameter $\gamma > 0$ is chosen so that $ \Lambda(\gamma) > 0 $, then the sequence $\{  \Theta_{\gamma} (x^t, y^t, z^t)\}$ is nonincreasing. 
\end{lemma}
\begin{remark}\label{rem-fun}When $F(x) = 0$, from \cref{alg:DYS} we have the variable $x^t = y^t = z^t$ , so the
energy function $\Theta_{\gamma}$ here is consistent with the decreasing function of Forward-Backward splitting in \cite{ABS}. When $H(x) = 0$, we also see that the energy function $\Theta_{\gamma}$ is the same as the merit function of DRS method in \cite{LP}. We also remark that we have $\Lambda(\gamma) \to +\infty$ when $\gamma \to 0$. Therefore, given $l \in \mathbb{R}$ and $L, \beta > 0$, $\Lambda(\gamma) > 0$ always holds if $\gamma >0$ is sufficiently small.
\end{remark}

Next we will formulate general conditions in terms of the input data of problem \eqref{model} which guarantee the boundedness of the sequence $\{(x^t, y^t, z^t) \}$ generated by \cref{alg:DYS}.
\begin{theorem}\label{bounded}
Let \cref{ass1} be satisfied and let the parameter $\gamma$ in \cref{alg:DYS} be such that $\Lambda(\gamma) > 0$. Suppose that the functions $F$, $G$ and $H$ are both bounded below and one of which is coercive. Then every sequence $\{(x^t, y^t, z^t)\}$ gernerated by \cref{alg:DYS} is bounded.
\end{theorem}
Proofs of \cref{decrease} and \cref{bounded} are given in \cref{app1}.\\

We now proceed to prove the first global convergence result for the \cref{alg:DYS}, which also gives the properties of the cluster point of sequence generated by \cref{alg:DYS}.
\begin{theorem}\label{clus-point}
(Global subsequential convergence). Let \cref{ass1} be satisfied and let the parameter $\gamma$ in \cref{alg:DYS} be such that $\Lambda(\gamma) > 0$. Then we have\\
(i) \begin{equation}\label{ineq15}
\lim_{t \to \infty} \| y^{t+1} - y^{t} \| = \lim_{t \to \infty} \| x^{t+1} - x^t \| = \lim_{t \to \infty} \| z^{t+1} - y^{t+1} \| = 0;
\end{equation}
(ii) Any cluster point $( x^*, y^*, z^*)$ of sequence $\{ (x^t, y^t, z^t) \} $ generated by \cref{alg:DYS} satisfies:
\begin{equation}\label{sta-point}
0 \in \nabla F(y^*) + \partial G(y^*) + \nabla H(y^*).
\end{equation}
\end{theorem}

Next, we will show the global convergence of the whole sequence generated by \cref{alg:DYS} under the additional assumption that the energy function $\Theta_{\gamma}$ is a KL function. In our proof, we will make use of the KL property; see \cref{KL}. This property has been used in many articles, such as \cite{LP, LLP, LPADMM, BCN, ABS, BZ}. In our analysis, we follow the similar line of these papers to prove the convergence of the sequence. 

In the following, we will show that if $\Theta_{\gamma}(x, y, z)$ is a KL function, then sequence $\{ (x^{t}, y^{t}, z^{t})\}_{t \geq 1}$ converges to a stationary point of the problem \eqref{model}. 
\begin{theorem}\label{whole-con}(Global convergence of the whole sequence)
Let \cref{ass1} be satisfied and let the parameter $\gamma$ in \cref{alg:DYS} be such that $\Lambda(\gamma) > 0$. Let $\{ (x^t, y^t, z^t) \}_{t \geq 1}$ be a sequence generated by \cref{alg:DYS} which has a cluster point. If $\Theta_{\gamma}$ is a KL function, then the following statements hold:
\begin{itemize}
\item[(i)]  The limit $\lim_{t \to \infty} \Theta_{\gamma}(x^t, y^t, z^t)$ exists and for any cluster point $(x^*, y^*, z^*)$ of the sequence $\{ (x^t, y^t, z^t) \}$ have 
\begin{equation}\label{ineq27}
\Theta^*:= \lim_{t \to \infty}\Theta_\gamma(x^t, y^t, z^t) = \Theta_\gamma (x^*, y^*, z^*);
\end{equation}
\item[(ii)] The sequence $\{ (x^t, y^t, z^t) \}_{t \geq 1}$ has finite length, that is,
\begin{equation}\label{ineq32}
\sum_{t \geq 1} \| x^{t+1} - x^t \| < +\infty;~~~~\sum_{t \geq 1} \| y^{t+1} - y^t \| < +\infty; ~~~~\sum_{t \geq 1} \| z^{t+1} - z^t \| < +\infty.
\end{equation}
Therefore, the whole sequence $\{(x^t, y^t, z^t) \}$ is convergent.
\end{itemize}
\end{theorem}

Finally, we  give eventual convergence rates of the nonconvex DYS method by examining the range of the exponent.
\begin{theorem}\label{conv-rate}(Eventual convergence rate)
Let the parameter $\gamma > 0$ be chosen such that $\Lambda(\gamma) > 0$ and $\{ x^t, y^t, z^t \}$ be a sequence generated by \cref{alg:DYS}. Suppose  $\{ x^t, y^t, z^t \}$ has a cluster point $(x^*, y^*, z^*)$. Suppose in addition that $F$, $H$ and $G$ are KL functions such that the $\varphi$ in \cref{KL} has the form $\varphi(s) = c s^{1-\theta}$ for some $\theta \in [0,1)$ and $c > 0$. Then, we have
\begin{itemize}
\item[(i)] If $\theta = 0$, then there exists $t_0 \geq 1$ such that for all $t \geq t_0,$ $0 \in \nabla F(z^t) + \partial G(z^t) + \nabla H(z^t)$;
\item[(ii)] If $\theta \in (0, \frac{1}{2}]$, then there exists $\eta \in (0,1)$ and $\kappa > 0$ so that $\textmd{dist}(0, \nabla F(z^t) + \partial G(z^t) + \nabla H(z^t)) \leq \kappa \eta^t$ for all large $t$;
\item[(iii)] If $\theta \in (\frac{1}{2}, 1)$, then there exists $\kappa > 0$ such that $\textmd{dist}(0, \nabla F(z^t) + \partial G(z^t) + \nabla H(z^t)) \leq \kappa t^{-\frac{1}{4\theta - 2}}$ for all large t. 
\end{itemize}
\end{theorem}
Please refer to \cref{app2} for proofs of \cref{clus-point}, \cref{whole-con} and \cref{conv-rate}.

\section{ Numerical examples}
\label{sec:examples} 
In this section, we implement DYS algorithm on low-rank matrix recovery and compressed sensing experiments, and compare numerical results with other classical algorithms. All experiments are run in MATLAB R2019a on a desktop computer equipped with a 4.0GHz 8-core AMD processor and 16GB memory. All the singular value decompose (SVD) involved in the experiments were conducted by using PROPACK coming in a MTLAB version.
\subsection{Low rank matrix recovery}
\label{4.1}
Low rank matrix recovery problem is a fundamental problem with many important applications in machine learning and signal processing. Over the years, many algorithms have been developed to solve this problem. A classical model of solving this problem is as follows
\begin{equation}\label{ineq47}
\min_{X \in \mathbb{R}^{m \times n}}  rank(X) ~~~~~s.t~~~ \mathcal{P}_{\Omega} (X) = \mathcal{P}_{\Omega}(M), 
\end{equation}
where $\Omega$ is the index set of matrix entries that are uniformly sampled, $\mathcal{P}_{\Omega}$ is the orthogonal projector onto the span of matrices vanishing outside of $\Omega$ so that the $(i, j)$th component of $\mathcal{P}_{\Omega}(X)$ is equal to $X_{ij}$ if $(i, j) \in \Omega$ and zero otherwise. However, the form \eqref{ineq47} is generally NP-hard and is also NP-hard to approximate \cite{MJCD}. There have been some important breakthroughs on this problem in recent years. In \cite{JMD}, the authors introduced the Singular Value Projection (SVP) algorithm which is based on projected gradient descent to tackle the following more robust formulation of \eqref{ineq47},
\begin{equation}\label{ineq48}
\min_{X} \frac{1}{2} \| \mathcal{P}_{\Omega}(X) - \mathcal{P}_{\Omega}(M) \|_{2}^{2} + \mathcal{I}_{\mathcal{C}(r)}(X),
\end{equation}
where $\mathcal{C}(r) := \{ X | rank(X) \leq r \},$~$\mathcal{I}_{\mathcal{C}(r)}(\cdot)$ denotes the indicator function of $\mathcal{C}(r)$. Specifically, in \cite{JMD} the algorithm to solve the problem $\eqref{ineq48}$ can be expressed as 
\begin{equation} \label{SVP}
(SVP) ~ \begin{cases}
Y^{t+1} = X^t - \eta_t \mathcal{P}_\Omega^T (\mathcal{P}_\Omega(X^t) - b),~~~~~~~~~~~~~~~~~~~~~~~\\
X^{t+1} = P_{C(r)}( Y^{t+1} ), 
\end{cases}
\end{equation} 
where $U_r,$ $\Sigma_{r}$, $V_{r}$ are the singular value decompose of $Y^{t+1}$. On the other hand, Cai, Cand$\grave{e}$s and Shen studied the tightest convex relaxation of the problem \eqref{ineq47}. They presented a singular value thresholding (SVT) algorithm for matrix completion, which may be expressed as 
\begin{equation} \label{SVT}
(SVT) ~ \begin{cases}
Y^{t+1}= \Sigma_{j=1}^{r_t}(\sigma_j^{t} - \tau)u_j^{t}v_j^{t}, \\

X^{t+1}_{ij}=~\begin{cases}
 0, ~~~~ & \textmd{if} ~ (i, j) \not\in \Omega,\\
X^{t}_{ij} + \delta(M_{ij} - Y^{t+1}_{ij}), ~~~~ & \textmd{if} ~ (i, j) \in \Omega, 
\end{cases}
\end{cases} 
\end{equation} 
where $U^{t}$, $\Sigma^{t}$, $V^{t}$ are the singular value decomposition of the matrix $Y^{t}$,  and $u_j^{t}, \sigma_j^{t}, v_{j}^{t}$ are corresponding singular vectors and singular values, then they showed that the sequence $X^{t} $ generated by the SVT algorithm \eqref{SVT} converges to the unique solution of an optimization problem, namely, 
\begin{equation}\label{ineq49}
\begin{aligned}
&\min~~~~\tau \| X \|_{*} + \frac{1}{2} \| X \|_{F}^{2},\\
&s.t.~~~~\mathcal{P}_{\Omega}(X) = \mathcal{P}_{\Omega}(M).
\end{aligned}
\end{equation} 
From the above we can see that when the SVT method \cite{CCS} solve the low-rank matrix recovery problem, the sequence actually converges to a the problem with an additional regularization term $\| \cdot \|_{F}^2$.  The numerical results (see \cite{CCS}) showed that this method is very efficient , which indicates that the additional  regularization term $\| \cdot \|_{F}^2$ have a good effect for this problem. In addition, the effectiveness of regularization terms $\| \cdot \|^2$ has also been demonstrated in some other nonconvex optimizations (see, e.g., \cite{LZT, ZH}). Therefore, we here use the DYS method to solve the problem \eqref{ineq48} with an additional regularization term $\frac{\lambda}{2} \| X \|_{2}^{2}$, that is, 
\begin{equation}\label{ineq51}
\min_{X} \frac{1}{2} \| \mathcal{P}_{\Omega}(X) - \mathcal{P}_{\Omega}(M) \|_{2}^{2} + \mathcal{I}_{\mathcal{C}(r)}(X) + \frac{\lambda}{2} \| X \|_{2}^{2},
\end{equation}
where $\lambda$ is the regularization parameter. Hence,  applying the DYS method to solving \eqref{ineq51} with $F = \frac{1}{2} \| \mathcal{P}_{\Omega}(X) - \mathcal{P}_{\Omega}(M) \|_{2}^{2}$, $G = \mathcal{I}_{\mathcal{C}(r)}(X)$ and $H = \frac{\lambda}{2} \| X \|_{2}^{2}$ gives the following algorithm:
\begin{equation}\label{DYSm}
(DYS) ~ \begin{cases}
 U^{t+1}=
\begin{cases}
 \frac{1}{1+\gamma}\left( X_{i, j}^{t} + \gamma M_{i, j}\right),  ~~~ & (i, j) \in \Omega, \\
 X_{i, j}^{t},  ~~~~ & (i, j)\notin \Omega, \\
\end{cases}
\\
\\

V^{t+1} = P_{C(r)} ((2 - \gamma\lambda) U^{t+1} - X^{t}),\\

\\
X^{t+1} = X^{t} + (V^{t+1} - U^{t+1}).
\end{cases}
\end{equation}
We now verify the assumptions on $F$, $G$ and $H$ in convergence theory of the algorithm \eqref{DYSm} in \cref{sec:convergence}:
\begin{itemize}
\item[1.] Since $\mathcal{P}_{\Omega}$ is the orthogonal projection, we can easily know that $F(X) = \frac{1}{2}\|\mathcal{P}_{\Omega}(X) - \mathcal{P}_{\Omega}(M)\|^{2}$ is smooth with a Lipschitz continuous gradient whose Lipschitz continuity modulus $L$ is 1. This verifies the \cref{ass1} $(a1)$;
\item[2.] For the function $G(X) = I_{C(r)}(X)$, the proximal mapping of $G$ exists and hence the \cref{ass1} $(a2)$ is satisfied;
\item[3.] Clearly, $H(X) = \frac{\lambda}{2}\| X \|^2$ has a Lipschitz continuous gradient and is a coercive function.
\end{itemize}

Here we recall that when $\lambda= $2, algorithm \eqref{DYSm} is the classical DRS method solving \eqref{ineq48}. For the DYS and DRS methods, we adapt the heuristics described in \cite{LP} to select the parameter $\gamma$ as follows:\\

{\it We initialize $\gamma = k * \gamma_0$ and update $\gamma$ as $\max \{ \frac{\gamma}{2}, 0.9999\cdot \gamma_0 \}$ whenever $\gamma > \gamma_0$, and the sequence satisfies either $\| y^t - y^{t-1} \| > 1000/t$ or $\| y \|_{\infty} > 1e10$.}\\

\noindent For DYS method, we take $L=1$, $l=0$ and $\beta = 1$ in \eqref{cond-lam}, and we can easily get $\gamma_0 = 0.15$ satisfying : $\Lambda(\gamma) > 0$ when $0< \gamma < 0.15$. We choose $\gamma$ for the DR method as in \cite{LP}. We set $k = 10^6$ for all algorithms. We note that although $k = 10^6$ is selected large here, $\gamma$ will eventually be less than $\gamma_0$ as the iteration number increases, which also guarantees the convergence of the algorithm according to \cref{sec:convergence}.\\


{\bf Set simulation data and parameters for experiments.} We generate $n \times n$ matrices of rank $r$ by sampling two $n \times r$ factors $M_{L}$ and $M_{R}$ independently, each having i.i.d. Gaussian entries, and setting $M=M_{L}M_{R}^{*}$ as suggested in \cite{CR}. The set of observed entries $\Omega$ is sampled uniformly at random among all sets of cardinality $m$. The sampling ratio is defined as $p := \frac{m}{n^2}$. We wish to recover a matrix with lowest rank such that its entries are equal to those of $M$ on $\Omega$. In all experiments, we use 
\begin{equation}\label{ineq52}
\frac{\|\mathcal{P}_{\Omega}(X^t - M) \|_{F}}{\| \mathcal{P}_{\Omega}(M)\|_{F}} < 1 \times 10^{-4}
\end{equation}
as a stop criterion, where $\| \cdot \|_{F}$ represents the Frobenius norm. We compute the relative error as follows:
\begin{equation}\label{ineq53}
relative~error ~= \frac{\|X^{opt} - M\|_{F}}{\|M\|_{F}}.
\end{equation}
Next, we give the specific parameters selection and all the parameters are chosen to guarantee the convergence and according to the lowest relative error. For the SVT method, the parameters $\tau = 5n$ and $\delta = 1.2p^{-1}$ are chosen as in \cite{CCS}. For the $SVP$ method, we set the parameter $\eta = \frac{1}{p\sqrt{t}}$ as in \cite{JMD}. 
In algorithm \eqref{DYSm}, we set $\lambda = 1.5 \times10^{-6}$. In the following, we display our experimental results. We recover the matrix of $rank = 10$ or $30$ in different sizes $n = 3000,~5000,~8000,~10000~or~12000$ under the sampling ratio $p = 0.05$ or $p = 0.08$. All of these results are averaged over five runs. 

\begin{table}[htbp]\footnotesize
\centering
\resizebox{\textwidth}{!}{
\begin{tabular}{|c|c|c|c|c|c|c|c|c|c|}
\hline
\multicolumn{1}{|c|}{rank}&\multicolumn{1}{|c|}{size}
&\multicolumn{4}{|c|}{ Average runtime(s) / iterations}&\multicolumn{4}{|c|}{Relative error ( $10^{-4} $)}\\
\hline
&& SVT & SVP& DRS & DYS & SVT & SVP& DRS & DYS \\ \cline{2-10}
& 3000 & 38/92 & 159/618 & 217/337 & \textbf{33/56} & 1.38 & 1.41 & 1.41 & \textbf{0.95}\\ \cline{2-10}

rank=10 & 5000 & 127/74 & 378/526 & 486/282 & \textbf{90/54} & 1.20 & 1.27 & 1.27 & \textbf{0.93}\\\cline{2-10}

 & 8000 & 308/63 & 865/474 & 1169/252 & \textbf{231/50} & 1.17 & 1.18  & 1.19 & \textbf{0.90} \\\cline{2-10}

&10000& 481/60 & 1270/457 & 1746/244 & \textbf{323/45} & 1.05 & 1.15 & 1.15 & \textbf{0.85}\\\cline{1-10}

& 3000 & 80/167 & 418/658 & 514/607 & \textbf{64/77} & 1.85 & 1.68 & 1.89 & \textbf{1.10}\\\cline{2-10}

rank=30 & 5000 & 224/111 & 837/750 & 2816/297 & \textbf{160/67} & 1.51 & 1.57 & 1.28 & \textbf{1.02}  \\\cline{2-10}

& 8000 & 473/86 & 1650/606 & 2028/324 & \textbf{389/63} & 1.30 & 1.39 & 1.35 & \textbf{1.01}\\\cline{2-10}

& 10000 & 728/78 & 2113/562 & 990/405 & \textbf{571/62} & 1.24 & 1.32 & 1.56 & \textbf{0.98}\\ \cline{1-10}
\end{tabular}
}
\caption{Results of the average runtime, number of iterations and relative error when $p = 0.08$.}\label{tab-mc1}
\end{table}

\begin{table}[htbp]\footnotesize
\centering
\resizebox{\textwidth}{!}{
\begin{tabular}{|c|c|c|c|c|c|c|c|c|c|}
\hline
\multicolumn{1}{|c|}{rank}&\multicolumn{1}{|c|}{size}
&\multicolumn{4}{|c|}{ Average runtime(s) / iterations}&\multicolumn{4}{|c|}{Relative error ( $10^{-4} $)}\\
\hline
&& SVT & SVP& DRS & DYS & SVT & SVP& DRS & DYS \\ \cline{2-10}
& 5000 & 93/91 & 842/986 & 506/541 & \textbf{70/75} & 1.34 & 1.42 & 1.40 & \textbf{0.96}\\ \cline{2-10}

rank=10 & 8000 & 287/120 & 1886/849 & 1300/455 & \textbf{151/60} & 1.06 & 1.27 & 1.28 & \textbf{0.90}\\\cline{2-10}

 & 10000 & 278/69 & 2819/804 & 1843/429 & \textbf{227/59} & 1.19 & 1.24  & 1.22 & \textbf{0.91}\\\cline{2-10}

&12000& 737/65 & 6659/774 & 3857/411 & \textbf{567/58} & 1.19 & 1.20 & 1.20 & \textbf{0.95} \\\cline{1-10}

& 5000 & 208/163 & 2305/1681 & 1807/1010 &\textbf{143/106}  & 1.82 & 1.92 & 1.82 & \textbf{1.08} \\\cline{2-10}

rank=30 & 8000 & 367/113 & 4348/1214 & 2819/675 & \textbf{305/82} & 1.46 & 1.59 & 1.55 & \textbf{1.00}  \\\cline{2-10}

& 10000 & 590/99 & 6058/1087 & 3820/591 & \textbf{405/74} & 1.37 & 1.48 & 1.46 & \textbf{0.96} \\\cline{2-10}

& 12000 & 1402/90 & 6659/774 & 7150/541 & \textbf{989/72} & 1.29 & 1.20 & 1.40 & \textbf{0.95} \\ \cline{1-10}
\end{tabular}
}
\caption{Results of the average runtime, number of iterations and relative error when $p = 0.05$.}\label{tab-mc2}
\end{table}

\cref{tab-mc1} and \cref{tab-mc2} compare the runtime, the number of iterations and relative error required by various methods for $rank=10$ and $30$ in different sizes of matrix for sampling ratio $p=0.08$ or $0.05$. Clearly, DYS methods is substantially faster than the SVT, SVP and DRS methods. In particular, we can see from \cref{tab-mc2} that DYS method has very good behavior when the matrix size $n$ is large and the sampling rate $p$ is low.
We can see from \cref{tab-mc1} and \cref{tab-mc2} that DYS method can always find the solutions with highest accuracy.

{\bf Real data.}
We now evaluate our algorithms on the Movie-Lens \cite{data} data set, which contains one million ratings for 3900 movies by 6040 users. Table \ref{tab-mc3} shows the RMSE (root mean square error) obtained by each method with different rank $r$.  For SVP, we take step size $\eta$ as in \cite{JMD}.  For the classical DR  splitting and DYS algorithm, we adopt a heuristic method to choose $\gamma$ as before with $k = 100$ and we choose $\lambda = 10^{-3}$ in DYS method. Since the rank of matrices obtained by SVT cannot be fixed, we here don't consider SVT method.  As shown in \cref{tab-mc3},  we can see that DYS method outperforms the SVP and DR methods in terms of both RMSE and relative errors and runtime.
\begin{table}[htbp]\footnotesize
\centering
\begin{tabular}{|c|c|c|c|c|c|c|c|c|c|}
\hline
\multicolumn{1}{|c|}{size}&\multicolumn{3}{|c|}{ RMSE}&\multicolumn{3}{|c|}{relative error } &\multicolumn{3}{|c|}{runtime(s) / iterations }  \\
\hline
&~SVP~& ~DRS ~& ~DYS ~& ~SVP~& ~DRS~ & ~DYS~& ~SVP~& ~DRS~ & ~DYS~  \\ \cline{2-10}
5 & 1.05 & 0.84 & \textbf{0.82} & 0.28 & 0.23 & \textbf{0.22} & 467/330 & 307/235 & \textbf{229/191} \\ \cline{1-10}

10 & 0.99 & 0.79 & \textbf{0.77} & 0.26 & 0.21 & \textbf{0.20} & 603/319 & 388/269 & \textbf{275/197}    \\\cline{1-10}

 15 & 0.96 & 0.76 & \textbf{0.70} & 0.25 & 0.20 & \textbf{0.19} & 724/317 & 455/289 & \textbf{345/232}  \\\cline{1-10}

20 & 0.93 &  0.72 & \textbf{0.68} & 0.24 & 0.19 & \textbf{0.18} & 789/315 & 595/360 & \textbf{450/264}    \\\cline{1-10}

25 & 0.91 & 0.68 &\textbf{0.66}  & 0.24 & 0.18 & \textbf{0.17}  & 972/314 & 769/361& \textbf{487/276}  \\\cline{1-10}

30 & 0.88 & 0.65 & \textbf{0.64} & 0.23 & 0.17 & \textbf{0.17}  & 1114/313 & 874/374 & \textbf{696/306}   \\\cline{1-10}
\end{tabular}
\caption{RMSE, relative error and runtime obtained by each method with different rank $r$.}\label{tab-mc3}
\end{table}

\subsection{Compressed sensing}
Compressed sensing (CS) is an important research field in signal processing and mathematical research. A fundamental problem in CS is to recover a sparse vector from a set of linear measurements. Over the past decade, great efforts have been made to explore efficient and stable algorithms to solve the basis pursuit problem and its associated $l_1$-regularized problem (also known as Lasso \cite{T}):
\begin{equation}\label{ineq40}
\min_{x} \frac{1}{2}\| Ax - b \|_{2}^{2} + \lambda \| x \|_1,
\end{equation}
where $\lambda > 0$ is a regularized parameter, $A \in \mathbb{R}^{m \times n}$ is a sensing matrix, $b \in \mathbb{R}^m / \{ 0 \}$ the measurement data. At present, there are many algorithms to solve this model, such as \cite{BPCPE, E, GO, WYGY, YZ, YOGD}. In \cite{BPCPE}, the authors solved the Lasso problem \cref{ineq40} by ADMM (which called the ADMM-Lasso). We give the details of the algorithm in the following:

\begin{algorithm}[htp]
\caption{ ADMM for solving \eqref{ineq40}.}
\label{alg:admm-lasso}
\begin{algorithmic}
\STATE{
Define $\epsilon > 0$ and $z^{0},~y^{0}$.\\
~~~~{\bf {for}} $k = 0, 1, 2, \dots,$ Maxit {\bf{do}}
\begin{subequations}\label{ineq41}
\begin{align}
&y^{k+1} = (A^{T} A + \rho I)^{-1} (A^{T}b + \rho (z^{k} - x^{k})),\\
&z^{k+1} = \mathcal{S}_{\frac{\lambda}{\rho}} ( y^{k+1} + \frac{x^{k}}{\rho}),\\
&x^{k+1} = x^{k} + \rho (y^{k+1} - z^{k+1}).
\end{align}
\end{subequations}
~~~~~~~~{\bf {end for.}}
}
\end{algorithmic}
\end{algorithm}
Later, the authors in \cite{YLHX, LYHX} applied the difference of $l_1$ and $l_2$ norms as a nonconvex and Lipschitz continuous metric to solve unconstrained CS problem. They showed that when the sensing matrix $A$ is ill-conditioned, such as an oversampled discrete cosign transform (DCT) matrix, the $l_{1-2}$ metric will better than existing nonconvex compressed sensing solvers. We present the model of \cite{YLHX} in the following:
\begin{equation}\label{ineq42}
\min_{x} \frac{1}{2} \| Ax - b\|_2^2 + \lambda (\| x \|_1 - \| x \|_2), 
\end{equation}
where $\lambda > 0$ is a regularized parameter, $A \in \mathbb{R}^{m \times n}$ is a sensing matrix, $b \in \mathbb{R}^m / \{ 0 \}$ is the measurement data. They employed the difference of the convex functions algorithm (DCA) to solve this model, and the algorithm is given below (see \cref{alg:DCAL12}).

\begin{algorithm}[htbp]
\caption{ DCA-$l_{1-2}$ for solving \eqref{ineq42}.}
\label{alg:DCAL12}
\begin{algorithmic}
\STATE{Define $\epsilon >0$ and set $y^{0}=0.$}

\STATE{{\bf {for}} $t = 0, 1, 2, \dots,$ Maxoit  {\bf {do}}\\
~~~~ Define $z^{0},~x^{0}$.\\
~~~~~~~~{\bf {for}} $k = 0, 1, 2, \dots,$ Maxit {\bf{do}}
\begin{subequations}\label{ineq43}
\begin{align}
&y^{k+1} = (A^{T} A + \rho I)^{-1} (A^{T}b + \lambda \frac{\| y^t \|}{\|y^t\|_2} + \rho (z^{k} - x^{k})),\\
&z^{k+1} = \mathcal{S}_{\frac{\lambda}{\rho}} ( y^{k+1} + \frac{x^{k}}{\rho}),\\
&x^{k+1} = x^{k} + \rho (y^{k+1} - z^{k+1}).
\end{align}
\end{subequations}
~~~~~~~~{\bf {end for.}}\\
~~~$y^{t} = y^{k+1}.$\\
~~{\bf {end for.}}
}
\end{algorithmic}
\end{algorithm}
If we take $F = \frac{1}{2} \| Ax - b\|_2^2$, $G = \lambda \| x \|_1$ and $H = -\lambda \| x \|_2$ in model \eqref{model}, it is easy to verify that the assumptions in the convergence theory of \cref{sec:convergence} are satisfied. Here, we use the DYS method to solve model \eqref{ineq42} (called $DYS-l_{1-2}$ see \cref{alg:DYSL12}) and compare it with the ADMM-Lasso (which solves the Lasso problem \eqref{ineq40} by ADMM) and the $DCA-l_{1-2}$ (which solves the problem \eqref{ineq42} by DCA). In the following, we give the specific details of experiments setting.


\begin{algorithm}[htbp]
\caption{ DYS-$l_{1-2}$ for solving \eqref{ineq42}.}
\label{alg:DYSL12}
\begin{algorithmic}
\STATE{
Define $\epsilon > 0$ and $x^{0}$.\\
~~~~{\bf {for}} $k = 0, 1, 2, \dots,$ Maxit {\bf{do}}
\begin{subequations}\label{eq01}
\begin{align}
&y^{k+1} = (A^{T} A + \frac{1}{\gamma} I)^{-1} (A^{T}b + \frac{1}{\gamma} x^{k}),\\
&z^{k+1} = \mathcal{S}_{\gamma \lambda} ( 2y^{k+1} + \gamma \lambda \frac{y^{k+1}}{\| y^{k+1} \|} - x^{k}),\\
&x^{k+1} = x^{k} + (z^{k+1} - y^{k+1}).
\end{align}
\end{subequations}
~~~~~~~~{\bf {end for.}}
}
\end{algorithmic}
\end{algorithm}

{\bf Set the sensing matrix.}  We will set that the matrix $A$ is an ill-conditioned DCT matrix. Such matrices are generated as follows:
$$
A_i = \frac{1}{\sqrt{m}} cos( 2i \pi \xi / F), ~~~~i = 1, \cdots, n, 
$$
where $\xi \in \mathbb{R}^{m} \sim \mathcal{U}( [0, 1]^m )$ whose components are uniformly and independently sampled from $[0,1]$ and $F \in \mathbb{N}$ is the refinement factor. In fact, it is the real part of the random partial Fourier matrix (see \cite{FL}). The number $F$ is bound up with the conditioning of $A$, in the sense that, the coherence of matrix $A$ (see Definition 2.2 in \cite{YLHX}) becomes large as $F$ increases. In our experiments, for $A \in \mathbb{R}^{m \times 2000}$ with $m \in \{ 80, 100, 120, 150, 200\}$, the coherence of $A$ always exceeds 0.99 when $F=10$ for all possible $m$. Although such sampled $A$ does not have a good restricted isometry property (RIP) in any case, it is still possible to recover the sparse vector $\bar{x}$ as long as its spikes are sufficiently separated. More specifically, the elements of $supp(\bar{x})$ are randomly chosen such that 
$$
\min_{i, j \in supp(\bar{x})} | i - j | \geq L.
$$
Here, $L$ is called the {\it minimum separation.} 

{\bf Select parameters for experiments.} We set $L=2F$ and implement our experiment as follows. After obtaining a sensing matrix as described above, we generate a test signal $\bar{x}$ of sparsity $s$, which supported on a random index set with independent and identically distributed Gaussian entries. Then we can calculate the measurement $b = A\bar{x}$ and apply it to every method to produce a reconstruction signal $x^*$. The reconstruction is considered a success if the relative error satisfy:
\begin{equation}\label{ineq44}
\frac{\| x^* - \bar{x} \|_{2}}{\| \bar{x} \|_{2}} < 10^{-4}.
\end{equation}
We run 100 independent experiments and record their corresponding success rates at different sparse levels, and we figure out the mean and standard deviations of the relative errors of all successful experiments.
According to \cite{BPCPE}, for ADMM-lasso in \cref{alg:admm-lasso}, we choose $\lambda = 10^{-6},~\beta=1,~\rho=10^{-5},~\epsilon^{abs} = 10^{-7},~\epsilon^{rel} = 10^{-5}$ and its maximum number of iteration $maxiter = 50000.$ According to \cite{YLHX}, for DCA-$l_{1-2}$ in \cref{alg:DCAL12}, we choose $\lambda = 10^{-5}$, $\epsilon^{abs} = 10^{-7}$, $\epsilon^{rel} = 10^{-5}$, maximum number of iterations of the outer loop and inner loop are $Maxoit = 10$ and $Maxit = 5000$. All parameters are selected according to the choice in \cite{YLHX, BPCPE}, which makes results of their experiments the best. Meanwhile, for the outer iteration in \cref{alg:DCAL12}, we adopted 
\begin{equation}\label{ineq45}
\frac{\| x^{k+1} - x^{k}\|_2}{\max\{\|x^k\|_{2}, 1\}} < 10^{-2}.
\end{equation}
For DYS-$l_{1-2}$ in algorithm \cref{alg:DYSL12}, we choose $\lambda = 10^{-5}~\epsilon^{abs} = 10^{-7},~\epsilon^{rel} = 10^{-5}$ and $Maxit = 50000$. For the choice of $\gamma$, we also use the heuristics as in \cref{4.1}.

According to \cite{BPCPE}, a stopping criterion for ADMM-Lasso, DYS-$l_{1-2}$ and the inner iteration of DCA-$l_{1-2}$ is given by 
\begin{equation}\label{ineq46}
\| r^k \|_2 \leq \sqrt{n} \epsilon^{abs} + \epsilon^{rel} \max \{ \| y^{k} \|_2, \| z^{k} \|_2\},~~~~\|s^{k}\|_2 \leq \sqrt{n} \epsilon^{abs} + \epsilon^{rel}\| x^{k} \|_2,
\end{equation}
where $r^k = y^k - z^k$, $s^k = \rho (z^k - z^{k-1})$ are primal and dual residuals at the $k$th iteration respectively. $\epsilon^{abs} > 0$ is an absolute tolerance and $\epsilon^{rel} > 0$ a relative tolerance.\\

{\bf Test results on highly coherent matrix.}
\cref{fCS1} shows the success rates of three different algorithms under various sparsity $s$ and various  sizes of $m$. We can see from the figure that the areas of the blue part corresponding to DCA-$l_{1-2}$ and DYS-$l_{1-2}$ are almost the same, and they are smaller than the area of the blue part corresponding to ADMM-Lasso. This means that the success rates of DYS-$l_{1-2}$ and DCA-$l_{1-2}$ are basically the same, but they are both better than ADMM-Lasso. 
\begin{figure}[htbp]
\centering
\includegraphics[width = 4in, height= 1.5in ] {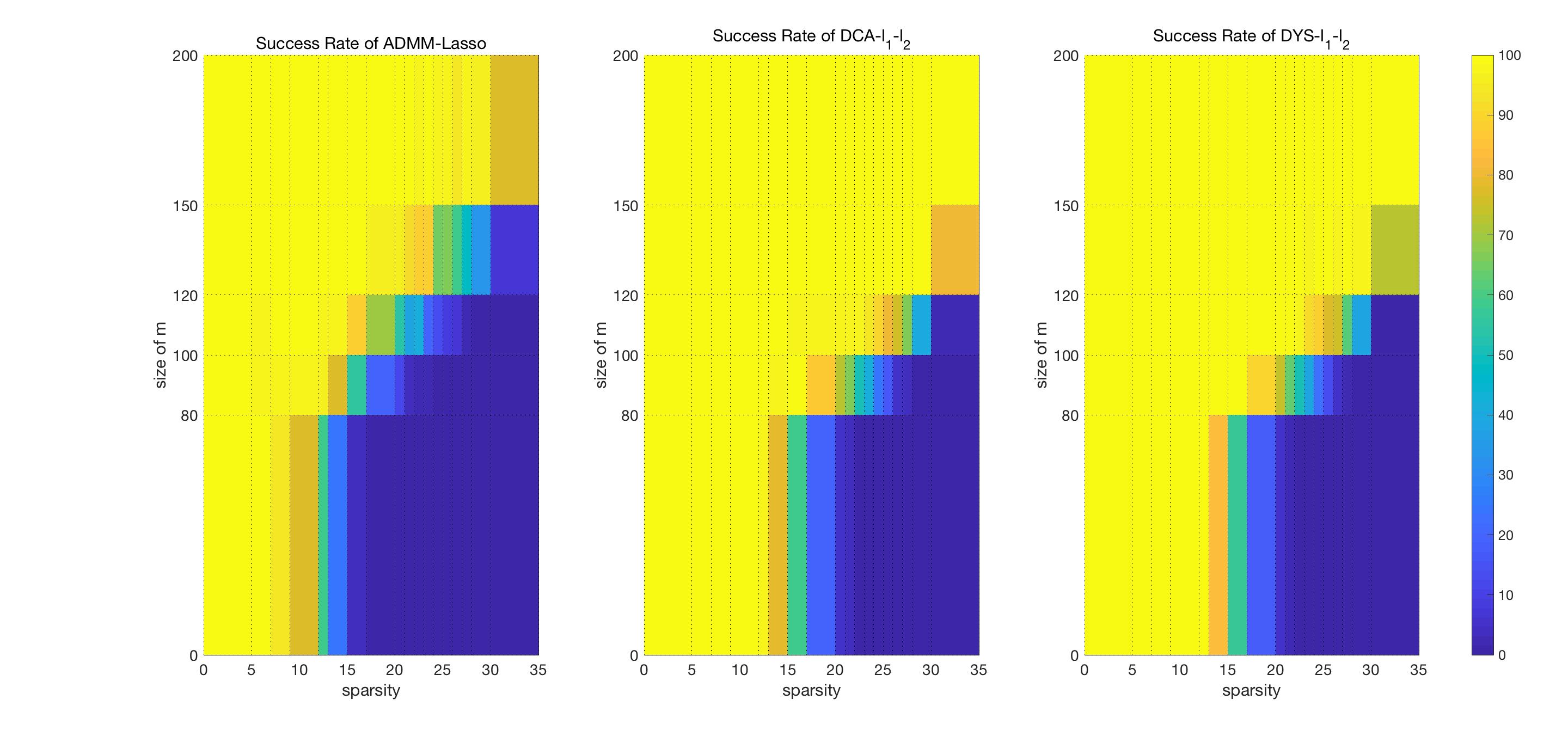}
\caption{Success rate of different methods}\label{fCS1}
\end{figure}

 \begin{table}[htbp]\footnotesize
\centering
\resizebox{\textwidth}{!}{
\begin{tabular}{|c|c|c|c|c|c|}
\hline
Algorithm & s=5 & s=9 &  s=15 & s=17 & s=20 \\ \cline{1-6}
ADMM-Lasso   &5.49/0.90 & 10.72/2.36& 29.40/19.02& 44.12/25.00& 75.76/28.89 \\ \cline{1-6}
DCA - $l_{1-2}$ & 5.07/0.41 & 9.09/0.71 & \textbf{15.39/1.33} &\textbf{17.17/0.60} & \textbf{28.42}/24.22 \\\cline{1-6}
DYS - $l_{1-2}$ & \textbf{5.00/0.00} & \textbf{9.08/0.28}& 16.00/1.39 & 19.29/3.36 & 33.03/\textbf{20.55}  \\\cline{1-6}
\end{tabular}
}
\caption{The average of the sparsity and the standard deviation when the noise level is 0.}\label{CS1}
\end{table}

\cref{CS1} shows the average of the sparsity and the standard deviation when the noise level is 0. We calculate the sparsity and relative error on truncated signal, that is, if the component is less than $5 \times 10^{-6}$, then we take the corresponding value to be 0. We can see from \cref{CS1} that the sparsity and the standard deviation of DCA-$l_{1-2}$ and DYS-$l_{1-2}$ are comparable, which are both smaller that ADMM-Lasso’s. However, from \cref{CS2},  the relative error of DYS-$l_{1-2}$ is smallest, which means that the solution given by DYS-$l_{1-2}$ is the most accurate. When the measurement data are added noises with different levels, \cref{CS3} shows the relative error of the signals recovered by ADMM-Lasso, DCA-$l_{1-2}$ and DYS-$l_{1-2}$. We can also see that the signal recovered by DYS-$l_{1-2}$ method is more accurate than the other two methods. Therefore, overall, DYS-$l_{1-2}$ performs better than the other two algorithms.

 \begin{table}[htbp]\footnotesize
\centering
\resizebox{\textwidth}{!}{
\begin{tabular}{|c|c|c|c|c|c|}
\hline
Algorithm & s=5 & s=9 &  s=15 & s=17 & s=20 \\ \cline{1-6}
ADMM-Lasso   &0.09/0.08 & 0.12/0.11 & 0.20/0.21& 0.20/0.19& \textbf{0.07/0.19} \\ \cline{1-6}
DCA - $l_{1-2}$ & 0.31/0.16 & 0.36/0.15 & 0.44/0.15 &0.43/0.14 & 0.42/0.21 \\\cline{1-6}
DYS - $l_{1-2}$ & \textbf{0.08/0.03} & \textbf{0.09/0.02}& \textbf{0.13/0.03} & \textbf{0.15/0.08} & 0.20/0.15 \\\cline{1-6}
\end{tabular}
}
\caption{The average of the relative error ($10^{-4}$)and the standard deviation when the noise level is 0.}\label{CS2}
\end{table}

\begin{table}[h!bp]\footnotesize
\centering
\resizebox{\textwidth}{!}{
\begin{tabular}{|c|c|c|c|c|c|c|}
\hline
\multicolumn{1}{|c|}{noise}&\multicolumn{1}{|c|}{Algorithm}
&\multicolumn{5}{|c|}{ relative error / standard deviation}\\
\hline
&& s=5 & s=9 & s=15 & s=17 & s=20  \\\cline{2-7}
&ADMM-Lasso & 0.5394~/~0.4077 & 0.5803~/~0.8540   & 0.6061~/~0.5302  & 0.5808~/~0.4760 & 0.6761~/~0.6827 \\\cline{2-7}

$\sigma$ =0.01 &DCA - $l_{1-2}$~& 0.5443~/~0.5124        & 0.5603~/~0.8244        & 0.4687~/~0.4500        & 0.5513~/~0.5822        & 0.6444~/~0.5413 \\\cline{2-7}

 &~DYS - $l_{1-2}$~&\textbf{0.2476~/~0.2799} &\textbf{0.2171~/~0.2483}&\textbf{0.2338~/~0.2389}&\textbf{0.3169~/~0.3426}&\textbf{0.3495~/~0.2303} \\\cline{1-7}

& ADMM-Lasso & 0.2384~/~0.2904         & 0.2297~/~0.2515        & 0.3512~/~0.3105        & 0.3804~/~0.2391        & 0.4538~/~0.3004 \\\cline{2-7}
$\sigma$ = 0.005 &~DCA - $l_{1-2}$  & 0.2168~/~0.3324         & 0.1969~/~0.3019        & 0.3002~/~0.3651        & 0.3117~/~0.2725        & 0.3636~/~0.3724 \\\cline{2-7}
 &~DYS - $l_{1-2}$~  &\textbf{0.0793~/~0.1189} &\textbf{0.0717~/~0.0743}&\textbf{0.1302~/~0.1226}&\textbf{0.1306~/~0.1229}&\textbf{0.2014~/~0.1907} \\\cline{1-7}

& ADMM-Lasso & 0.0305~/~0.0320         & 0.0419~/~0.0516        & 0.1025~/~0.0997        & 0.1174~/~0.1286        & 0.2798~/~0.2170 \\\cline{2-7}
$\sigma$ = 0.001 &~DCA - $l_{1-2}$  & 0.0185~/~0.0241         & 0.0249~/~0.0392        & 0.0464~/~0.0520        & 0.0447~/~0.0771        & 0.0864~/~0.2067 \\\cline{2-7}
&~DYS - $l_{1-2}$~  &\textbf{0.0081~/~0.0042} &\textbf{0.0077~/~0.0047}&\textbf{0.0105~/~0.0068}&\textbf{0.0126~/~0.0080}&\textbf{0.0403~/~0.2027} \\\cline{1-7}

& ADMM-Lasso & 0.0197~/~0.0375         & 0.0192~/~0.0211        & 0.0351~/~0.0501        & 0.0731~/~0.1067        & 0.1940~/~0.1869 \\\cline{2-7}
$\sigma$ = 0.0005 &~DCA - $l_{1-2}$  & 0.0116~/~0.0210         & 0.0086~/~0.0070        & 0.0086~/~0.0070      & 0.0135~/~0.0217        & 0.0462~/~0.1261 \\\cline{2-7}
&~DYS - $l_{1-2}$~  &\textbf{0.0031~/~0.0014} &\textbf{0.0035~/~0.0002}&\textbf{0.0040~/~0.0017}&\textbf{0.0047~/~0.0025}&\textbf{0.0077~/~0.0128} \\\cline{1-7}

\end{tabular}
}
\caption{The average of the relative error and the standard deviation with different noise levels.}\label{CS3}
\end{table}

\section{Concluding remarks}
\label{sec:conclude}
In this paper, we employ a three-operator splitting proposed by Davis and Yin (called DYS)  to  resolve two kinds of nonconvex problems in sparsity regularization: sparse signal recovery and low rank matrix recovery. We first study the convergence behavior of  Davis-Yin splitting algorithm in nonconvex setting. By constructing a new energy function associated with Davis-Yin method, we prove the global convergence and establish local convergence rate of the Davis-Yin splitting method when the parameter $\gamma$ is less than a computable threshold and the sequence generated has a cluster point. We also show the boundedness of the sequence generated by Davis-Yin splitting method when some sufficient conditions are satisfied, thus the existence of cluster points. Finally,  we show some numerical experiments to compare the DYS algorithm with some classical efficient algorithms for sparse signal recovery and low rank matrix completion. The numerical experiments indicate that the Davis-Yin splitting is significantly better than these methods.\\

\appendix
\section{Proofs of \cref{decrease} and \cref{bounded}}
\label{app1}

To prove Lemma 3.1 and Theorem 3.5, we first need the following two lemmas. The proof of Lemma A.1 is very easy, we omit it here.
\begin{lemma}\label{x-bounded} 
Suppose $F$ satisfies (a1) in \cref{ass1}. Then the sequence  $\{(x^t, y^t, z^t)\}$ generated by \cref{alg:DYS} satisfies 
\begin{equation}\label{xineq}
\| x^{t} - x^{t-1} \| \leq ( 1 + \gamma L ) \| y^{t+1} - y^{t} \|.
\end{equation}
\end{lemma} 

\begin{lemma}\label{eq-rela}
Let $a,~b,~c,~d \in \mathbb{R}^{n}$. Then we have 
\begin{equation}\label{eq100}
\begin{aligned}
&\| 2a - b - c - d \|^2 - \| a - c - d\|^2\\
& = (\| a - c \|^2 - \| b - c \|^2) + 2 \| a - b \|^2 + 2 \langle d, b - a \rangle.
\end{aligned}
\end{equation} 
\end{lemma}
\begin{proof}The proof is basic, it just requires some simple identities.
\begin{equation}
\begin{aligned}
&\| 2a - b - c - d \|^2 - \| a - c - d \|^2\\
&=\| a - b + ( a - c - d) \|^2 - \| a - c - d \|^2\\
&= \| a - b \|^2 + 2 \langle a - c - d, a - b \rangle\\
&= \| a - b \|^2 + 2 \langle a - c, a - b \rangle + 2\langle d, b - a \rangle\\
&= \| a - b \|^2 + \left( \| a - c \|^2 + \| a - b \|^2 - \| b - c\|^2 \right) + 2 \langle d, b - a \rangle\\
&= \left( \| a - c \|^2 - \| b - c \|^2 \right) + 2 \| a - b \|^2 + 2 \langle d, b - a \rangle.
\end{aligned}
\end{equation}
So we get the conclusion.
\end{proof}

\vskip0.10in
{\it Proof of \cref{decrease}.}
We will show first that:
\begin{equation}\label{eq200}
\begin{aligned}
&F(y^{t+1}) + G(z^{t+1}) + \frac{1}{2\gamma} \| 2y^{t+1} - z^{t+1} - x^{t+1} - \gamma \nabla H(y^{t+1})\|^2\\
 &~~~~~-\frac{1}{2\gamma} \| x^{t+1} - y^{t+1} + \gamma \nabla H(y^{t+1}) \|^2 - \frac{1}{\gamma}\| y^{t+1} - z^{t+1} \|^2\\
& \leq F(y^t) + G(z^t) + \frac{1}{2\gamma} \| 2y^t - z^t - x^t - \gamma \nabla H(y^t) \|^2 - \frac{1}{2\gamma}\| x^t - y^t + \gamma \nabla H(y^t) \|^2\\
& ~~~~~- \frac{1}{\gamma}\| y^t - z^t \|^2 + \langle \nabla H(y^{t+1}), z^t - y^{t+1} \rangle - \langle \nabla H(y^t), z^t - y^t \rangle \\
& ~~~~~ + \frac{1}{\gamma} \|y^{t+1} - z^t \|^2- \frac{1}{2}(\frac{1}{\gamma} - l) \| y^{t+1} - y^t \|^2
\end{aligned}
\end{equation}
and provide afterwards an upper estimate for the terms $\langle \nabla H(y^{t+1}), z^t - y^{t+1} \rangle - \langle \nabla H(y^t), z^t - y^t \rangle$ and $\frac{1}{\gamma}\| y^{t+1} - z^t \|^2.$

Since $F + \frac{1}{2\gamma}\| x^t - \cdot \|^2$ is a strongly convex function with modulus $\frac{1}{\gamma} - l$ and $y^{t+1}$ is a minimizer of \eqref{algy}, we obtain
\begin{equation}\label{eq201}
F(y^{t+1}) + \frac{1}{2\gamma}\|y^{t+1} - x^t \|^2 \leq F(y^t) + \frac{1}{2\gamma}\|y^t - x^t \|^2 - \frac{1}{2}(\frac{1}{\gamma} - l) \| y^{t+1} - y^t \|^2.
\end{equation}
From \eqref{algz}, we have
\begin{equation}\label{eq202}
\begin{aligned}
&G(z^{t+1}) + \frac{1}{2\gamma}\| z^{t+1} - 2y^{t+1} + \gamma \nabla H(y^{t+1}) + x^t \|^2\\
&\leq G(z^t) + \frac{1}{2\gamma}\| z^t - 2y^{t+1} +\gamma \nabla H(y^{t+1})  + x^t \|^2.
\end{aligned}
\end{equation}
 Adding \eqref{eq201} and \eqref{eq202} yields
 \begin{equation}\label{eq203}
 \begin{aligned}
& F(y^{t+1}) + G(z^{t+1}) + \frac{1}{2\gamma}\| 2y^{t+1} - z^{t+1} - x^t - \gamma \nabla H(y^{t+1}) \|^2\\
&~~~~~~~~~~~~~~~+\frac{1}{2\gamma}\|y^{t+1} - x^t \|^2\\
&\leq F(y^t) + G(z^t) + \frac{1}{2\gamma} \|2y^{t+1} - z^t - x^t - \gamma \nabla H(y^{t+1}) \|^2\\
&~~~~~~~~~~~~~~~+\frac{1}{2\gamma}\| y^t - x^t \|^2 - \frac{1}{2}(\frac{1}{\gamma} - l)\| y^{t+1} - y^t \|^2. 
\end{aligned}
\end{equation}
On the other hand, by applying some elementary identities and \eqref{algx} we also have 
\begin{equation}\label{eq204}
\begin{aligned}
&\| 2y^{t+1} - z^{t+1} - x^t - \gamma \nabla H(y^{t+1}) \|^2\\
&=\| 2y^{t+1} - z^{t+1} - x^{t+1} - \gamma\nabla H(y^{t+1}) \|^2\\
&~~~+2 \langle 2y^{t+1} - x^{t+1} - z^{t+1} - \gamma \nabla H(y^{t+1}), x^{t+1} - x^t \rangle + \| x^{t+1} - x^t \|^2\\
& =\| 2y^{t+1} - x^{t+1} - z^{t+1} - \gamma \nabla H(y^{t+1}) \|^2 + 2\langle 2y^{t+1}  - 2z^{t+1}, x^{t+1} - x^t \rangle \\
&~~~+ 2\langle y^{t+1} - x^t - \gamma \nabla H(y^{t+1}), x^{t+1} - x^t \rangle + \| x^{t+1} - x^t \|^2.
\end{aligned}
\end{equation}
Note that, by \eqref{algx} we have
\begin{equation}\label{eq205}
2 \langle 2y^{t+1} - 2z^{t+1}, x^{t+1} - x^t \rangle = -4 \| x^{t+1} - x^t \|^2.
\end{equation}
By the elementary identity $2\langle a, b \rangle = -(\| a - b \|^2 - \| a \|^2 - \| b \|^2)$, we have
\begin{equation}\label{eq206}
\begin{aligned}
&2 \langle y^{t+1} - x^t - \gamma \nabla H(y^{t+1}), x^{t+1} - x^t \rangle\\
&= - \left( \| y^{t+1} - x^{t+1} - \gamma \nabla H(y^{t+1}) \|^2 - \| y^{t+1} - x^t - \gamma \nabla H(y^{t+1})\|^2 - \| x^{t+1} - x^t \|^2 \right).
\end{aligned}
\end{equation}
Substituting \eqref{eq205} and \eqref{eq206} into \eqref{eq204}, we get
\begin{equation}\label{eq207}
\begin{aligned}
&\| 2y^{t+1} - z^{t+1} - x^t - \gamma \nabla H(y^{t+1}) \|^2\\
&= \| 2y^{t+1} - z^{t+1} - x^{t+1} - \gamma \nabla H(y^{t+1}) \|^2 - \| y^{t+1} - x^{t+1} - \gamma \nabla H(y^{t+1})\|^2\\
&~~~+ \|y^{t+1} - x^t - \gamma \nabla H(y^{t+1}) \|^2 - 2\| x^{t+1} - x^t \|^2.
\end{aligned}
\end{equation}
Combining \eqref{eq203} and \eqref{eq207}, and then using \cref{eq-rela}, we obtain
\begin{equation}\label{eq208}
\begin{aligned}
&F(y^{t+1}) + G(z^{t+1}) + \frac{1}{2\gamma}\| 2y^{t+1} - z^{t+1} - x^{t+1} - \gamma \nabla H(y^{t+1})\|^2\\
&~~~-\frac{1}{2\gamma}\| x^{t+1} - y^{t+1} + \gamma \nabla H(y^{t+1})\|^2 - \frac{1}{\gamma}\| y^{t+1} - z^{t+1} \|^2\\
&\leq F(y^t) + G(z^t) + \frac{1}{2\gamma}\| 2y^{t+1} - z^t - x^t - \gamma \nabla H(y^{t+1}) \|^2 + \frac{1}{2\gamma}\| y^t - x^t \|^2\\
& ~~~- \frac{1}{2\gamma}\| y^{t+1} - x^t - \gamma \nabla H(y^{t+1})\|^2 - \frac{1}{2\gamma}\| y^{t+1} - x^t \|^2 - \frac{1}{2}(\frac{1}{\gamma} - l)\|y^{t+1} - y^t \|^2\\
&= F(y^t) + G(z^t) - \frac{1}{2\gamma}\| z^t - x^t \|^2 + \langle \nabla H(y^{t+1}), z^t - y^{t+1} \rangle + \frac{1}{\gamma} \| y^{t+1} - z^t \|^2\\
& ~~~+ \frac{1}{2\gamma}\| y^t - x^t \|^2 - \frac{1}{2}(\frac{1}{\gamma} - l) \| y^{t+1} - y^t \|^2\\
& = F(y^t) + G(z^t) + \frac{1}{2\gamma} \left( \| y^t - x^t \|^2 - \| z^t - x^t \|^2 \right) + \langle \nabla H(y^t), z^t - y^t \rangle\\
& ~~~+ \frac{1}{\gamma}\| y^{t+1} - z^t \|^2 + \langle \nabla H(y^{t+1}), z^t - y^{t+1} \rangle - \langle \nabla H(y^t), z^t - y^t \rangle \\
&~~~- \frac{1}{2}(\frac{1}{\gamma} - l) \| y^{t+1} - y^t \|^2\\
& = F(y^t) + G(z^t) + \frac{1}{2\gamma}\| 2y^t - z^t - x^t - \gamma \nabla H(y^t) \|^2 -\frac{1}{2\gamma} \| x^t - y^t + \gamma \nabla H(y^t) \|^2\\
& ~~~+ \langle \nabla H(y^{t+1}), z^t - y^{t+1} \rangle - \langle \nabla H(y^{t}), z^t - y^t \rangle + \frac{1}{\gamma} \| y^{t+1} - z^t \|^2 \\
& ~~~ - \frac{1}{\gamma}\| y^t - z^t \|^2 - \frac{1}{2}(\frac{1}{\gamma} - l) \| y^{t+1} - y^t \|^2.
\end{aligned}
\end{equation}
This proves the \eqref{eq200}.

Next, we will focus on estimating $\langle \nabla H(y^{t+1}), z^t - y^{t+1} \rangle - \langle \nabla H(y^t), z^t - y^t \rangle$. According to the descent lemma we have
\begin{equation}\label{eq209}
\begin{aligned}
&\langle \nabla H(y^{t+1}), z^t - y^{t+1} \rangle - \langle \nabla H(y^{t}), z^t - y^t \rangle\\
&=\langle \nabla H(y^{t+1}) - \nabla H(y^t), z^t - y^{t+1} \rangle - \langle \nabla H(y^t), y^{t+1} - y^t \rangle\\
&\leq H(y^{t}) - H(y^{t+1}) + \frac{\beta}{2} \| y^{t+1} - y^t \|^2 + \langle \nabla H(y^{t+1}) - \nabla H(y^t), z^t - y^{t+1} \rangle\\
&\leq H(y^t) - H(y^{t+1}) + \frac{\beta}{2}\| y^{t+1} - y^t \|^2 + \frac{\beta}{2} \|y^{t+1} - y^t \|^2 + \frac{\beta}{2}\|y^{t+1} - z^t \|^2.
\end{aligned}
\end{equation}

Finally, we only need to estimate $\| y^{t+1} - z^t \|^2.$
From \eqref{opti-1},
\begin{equation}\label{eq210}
\nabla (F + \frac{l}{2} \| \cdot \|^2) (y^{t+1}) = \frac{1}{\gamma}(x^t - y^{t+1}) + l y^{t+1}.
\end{equation}
Note that $F + \frac{l}{2} \| \cdot \|^2$ is a convex function by assumption, using the monotonicity of gradient of a convex function, we have,
\begin{equation}\label{eq211}
\langle \left( \frac{1}{\gamma}(x^t - y^{t+1}) + l y^{t+1} \right) - \left( \frac{1}{\gamma}(x^{t-1} - y^t) + l y^t \right), y^{t+1} - y^t \rangle \geq 0,
\end{equation}
which gives 
\begin{equation}\label{eq212}
\langle y^{t+1} - y^t, x^t - x^{t-1} \rangle \geq (1 - \gamma l) \| y^{t+1} - y^t \|^2.
\end{equation}
Therefore, by \eqref{algx}, \eqref{eq212} and \eqref{eq-rela}, we have 
\begin{equation}\label{eq213}
\begin{aligned}
&\| y^{t+1} - z^t \|^2\\
&= \| y^{t+1} - y^{t} + y^t - z^t \|^2\\
&= \| y^{t+1} - y^{t} - (x^t - x^{t-1}) \|^2\\
&\leq \| y^{t+1} - y^t \|^2 - 2 \langle y^{t+1} - y^t, x^t - x^{t-1} \rangle + \| x^t - x^{t-1} \|^2\\
&\leq (-1 + 2\gamma l) \| y^{t+1} - y^t \|^2 + \| x^t - x^{t-1} \|^2\\
&\leq [(-1 + 2 \gamma l) + (1 + \gamma L)^2] \| y^{t+1} - y^t \|^2.
\end{aligned}
\end{equation}
By combining \eqref{eq200}, \eqref{eq209} and \eqref{eq213}, the desired conclusion follows.\hfill$\square$ \\

{\it Proof of \cref{bounded}.}
From \cref{ass1} (a1), there exists $\zeta^{*} > - \infty$ such that
\begin{equation}\label{ineq10}
\begin{aligned}
\zeta^{*} &\leq F\left(  x - \frac{1}{L} \nabla F(x) \right) \\
&\leq F(x) + \left\langle \nabla F(x), \left( x - \frac{1}{L} \nabla F(x) \right)  - x \right\rangle + \frac{L}{2} \left\|  \left( x - \frac{1}{L} \nabla F(x) \right) - x  \right\|^{2} \\ 
&= F(x) - \frac{1}{2L} \| \nabla F(x) \|^{2}.
\end{aligned}
\end{equation}  
Similarly, from \cref{ass1}(a3), there exists $\eta^{*} > - \infty$ such that
\begin{equation}\label{ineq11}
\eta^{*} \leq H\left(  x - \frac{1}{\beta} \nabla H(x) \right) \leq H(x) - \frac{1}{2\beta} \| \nabla H(x) \|^{2}.\\
\end{equation} 
By \eqref{opti-1}, we have that for any $t \geq 1$
\begin{equation}\label{ineq13}
\| x^{t-1} - y^{t} \|^{2} = \gamma^{2} \| \nabla F(y^{t})\|^{2}.
\end{equation}
By Cauchy-Schwarz inequality and \eqref{algx}, 
 \begin{equation}\label{ineq1000}
\begin{aligned}
\left\langle \nabla H( y^{t} ), z^{t} - y^{t} \right\rangle &\geq  - \frac{1}{4\beta} \| \nabla H( y^{t} ) \|^{2} - \beta \| z^{t} - y^{t} \|^{2}\\
& \geq - \frac{1}{4\beta} \| \nabla H( y^{t} ) \|^{2} - \beta \| x^{t} - x^{t-1} \|^{2}\\
&\geq  - \frac{1}{4\beta} \| \nabla H( y^{t} ) \|^{2} - 2\beta \| x^{t} - y^{t} \|^2 -2\beta \| y^{t} - x^{t-1} \|^{2}.\\
\end{aligned}
\end{equation}
This together with \eqref{ineq10}, \eqref{ineq11}, \eqref{ineq13} and \cref{eq-rela} yields that
 \begin{equation}\label{ineq14}
\begin{aligned}
&\Theta_{\gamma}( x^{1}, y^{1}, z^{1} ) \geq \Theta_{\gamma}( x^{t}, y^{t}, z^{t} )\\
&=F(y^t) + G(z^t) + H(y^t) + \frac{1}{2\gamma}\| 2y^t - z^t - x^t - \gamma \nabla H(y^t) \|^2\\
&~~~ - \frac{1}{2\gamma}\| x^t - y^t + \gamma \nabla H(y^t) \|^2 - \frac{1}{\gamma}\| y^t - z^t \|^2\\
&= F(y^{t}) + G(z^{t}) + H( y^{t} ) + \frac{1}{2\gamma} \| x^{t} - y^{t} \|^{2} - \frac{1}{2\gamma} \| x^{t} - z^{t} \|^{2} + \left\langle \nabla H( y^{t} ), z^{t} - y^{t} \right\rangle \\
&= F(y^{t}) + G(z^{t}) + H( y^{t} ) - \frac{1}{2\gamma} \| x^{t-1} - y^{t} \|^{2} + \frac{1}{2\gamma} \| x^{t} - y^{t} \|^{2} + \left\langle \nabla H( y^{t} ), z^{t} - y^{t} \right\rangle \\
&\geq F(y^{t}) + G(z^{t}) + H( y^{t} ) - \frac{1}{4\beta} \| \nabla H( y^{t} ) \|^2 - ( \frac{1}{2\gamma} + 2\beta ) \| x^{t-1} - y^{t} \|^{2} \\
&~~~+ ( \frac{1}{2\gamma} - 2\beta ) \| x^{t} - y^{t} \|^{2}\\
&\geq  F(y^{t}) + G(z^{t}) + H( y^{t} ) - \frac{1}{4\beta} \| \nabla H( y^{t} ) \|^2 - ( \frac{1}{2\gamma} + 2\beta )\gamma^{2} \| \nabla F( y^{t} ) \|^{2} \\
&~~~+ ( \frac{1}{2\gamma} - 2\beta ) \| x^{t} - y^{t} \|^{2}\\
&\geq \mu F(y^{t}) + (1 - \mu)F(y^{t}) - \frac{1 - \mu}{2L} \| \nabla F(y^{t}) \|^{2} + \left[ \frac{1 - \mu}{2L} - ( \frac{1}{2\gamma} + 2\beta ) \gamma^{2} \right] \| \nabla F(y^{t}) \|^{2}\\
&~~~~+ H(y^{t}) - \frac{1}{4\beta} \| \nabla H(y^{t}) \|^2 + G( z^{t} ) + \left( \frac{1}{2\gamma} - 2\beta \right) \| x^{t} - y^{t} \|^{2}\\
&\geq \mu F(y^{t}) + (1 - \mu) \zeta^{*} + \left[ \frac{1 - \mu}{2L} - ( \frac{1}{2\gamma} + 2\beta ) \gamma^{2} \right] \| \nabla F(y^{t}) \|^{2} + \frac{1- 4\gamma\beta}{2\gamma} \| x^{t} - y^{t} \|^{2}\\
&~~~~ +G(z^{t}) + \nu H(y^{t}) + ( 1 - \nu ) H( y^{t} ) - \frac{1 - \nu }{2\beta} \| \nabla H(y^{t}) \|^2 + \frac{1 - 2\nu}{4\beta} \| \nabla H(y^{t}) \|^2 \\
&\geq \mu F(y^{t}) + ( 1- \mu )\zeta^{*} + \left[ \frac{1 - \mu}{2L} - ( \frac{1}{2\gamma} + 2\beta ) \gamma^{2} \right] \| \nabla F(y^{t}) \|^{2} + \frac{1- 4\gamma\beta}{2\gamma} \| x^{t} - y^{t} \|^{2}\\
&~~~~+ G( z^{t} ) + \nu H(y^{t}) + ( 1 - \nu ) \eta^{*} + \frac{1 - 2\nu }{4\beta} \| \nabla H(y^{t}) \|^2,
\end{aligned}
\end{equation}
where we can choose $\gamma > 0$ small and $\mu,~\nu \in (0,1)$ such that
$$
\frac{1 - \mu}{2L} - ( \frac{1}{2\gamma} + 2\beta ) \gamma^{2},~\frac{1 - 2\nu }{4\beta},~\frac{1- 4\gamma\beta}{2\gamma} > 0.
$$
In the following, we divide into two cases:\\
{\bf Case 1.} $G$ is coercive. It is easy to see from \eqref{ineq14} that $\{ z^t \}$, $\{ \nabla F(y^t) \}$ and $\{ x^t - y^t\}$ are all bounded. So we can get from \eqref{ineq13} that $\{ y^t - x^{t-1}\}$ is bounded which implies that $\{ x^{t} - x^{t-1}\}$ is also bounded. Meanwhile, using \eqref{algx}, we can obtain that $\{ z^t - y^t\}$ is bounded. Thus $\{ y^t \}$ is bounded, because we have shown that $\{ z^t \}$ is bounded. Therefore, we can see that $\{ x^{t} \}$ is bounded by the boundedness of $\{ x^t - y^t \}$.\\
{\bf Case 2.} $F$ or $H$ is coercive. We can immediately get that $\{ y^t \}$ and $\{ x^t - y^t\}$ are bounded. Hence, $\{ x^t \}$ is also bounded. Now, the boundedness of $\{ z^t \}$ follows from \eqref{algx}. \hfill$\square$ \\

\section{Proofs of \cref{clus-point}, \cref{whole-con} and \cref{conv-rate}}
\label{app2}

\vskip0.10in
{\it Proof of \cref{clus-point}.}
Summing  \eqref{de-ineq}  from $t = 1$ to $N - 1 \geq 1$, we  get
\begin{equation}\label{sum}
\Theta_{\gamma} ( x^{N}, y^{N}, z^{N} ) - \Theta_{\gamma} ( x^{1}, y^{1}, z^{1} ) \leq -\Lambda(\gamma) \sum_{t=1}^{N} \| y^{t+1} - y^{t} \|^{2}.
\end{equation}
Suppose that $( x^{*}, y^{*}, z^{*} )$ is a cluster point of sequence $\{ x^{t}, y^{t}, z^{t} \}$, that is, there exists a convergent subsequence $\{ x^{t_j}, y^{t_j}, z^{t_j} \}$, such that
$$
\lim_{ j \to \infty} ( x^{t_{j}}, y^{t_{j}}, z^{t_{j}} ) = (x^{*}, y^{*}, z^{*} ).
$$
Since $\Theta_{\gamma}$ is a lower semi-continuious function and $F$, $G$ are both proper functions, we can take limit with $j \to \infty$ when $N = t_{j}$ in \eqref{sum}, 
\begin{equation}\label{sumineq}
- \infty < \Theta_{\gamma}( x^{*}, y^{*}, z^{*} ) - \Theta_{\gamma}( x^{1}, y^{1}, z^{1} ) \leq -\Lambda(\gamma) \sum_{t=1}^{\infty} \| y^{t+1} - y^{t} \|^{2}. 
\end{equation}
This implies that $lim_{t \to \infty} \| y^{t+1} - y^{t} \|^{2} = 0$. Combining with \cref{x-bounded} and \eqref{algx}, we obtain 
$\lim_{t \to \infty} \| x^{t+1} - x^{t} \| = \lim_{t \to \infty} \| z^{t+1} - y^{t+1} \| = 0$. Thus we get the desired conclusion $(i)$. 

We next prove $(ii)$. Firstly, by \eqref{algx}, we obtain further that $lim_{t \to \infty}\| z^{t+1} - z^{t} \| = 0$. Let $( x^{*}, y^{*}, z^{*} )$ be a cluster point of $\{ ( x^{t}, y^{t}, z^{t} )\}_{t \geq 1}$, assume that $\{ ( x^{t_{j}}, y^{t_{j}}, z^{t_{j}} )\}$ is a convergent subsequence such that
\begin{equation}\label{ineq16}
\lim_{t \to \infty} ( x^{t_j}, y^{t_j}, z^{t_j}) = (x^*, y^*, z^*).
\end{equation}
Then 
\begin{equation}\label{ineq17}
lim_{j \to \infty} ( x^{t_{j}}, y^{t_{j}}, z^{t_{j}} ) = lim_{j \to \infty} ( x^{t_{j - 1}}, y^{t_{j - 1}}, z^{t_{j - 1}} ) = ( x^{*}, y^{*}, z^{*} ). 
\end{equation}
Moreover, using the fact that $z^t$ is the minimizer in \eqref{algz}, we have
\begin{equation}\label{ineq18}
G( z^{t} ) + \frac{1}{2\gamma} \| z^{t} - ( 2 y^{t} - \gamma \nabla 
H(y^t) - x^{t-1}) \|^{2} \leq G( z^{*} ) + \frac{1}{2\gamma} \| z^{*} - (2 y^{t} - \gamma \nabla H(y^t) -x^{t-1}) \|^{2}. 
\end{equation}
Taking  limit along the subsequence $\{t^j\}$ and using \eqref{ineq17} yields
\begin{equation}\label{ineq19}
\lim \sup_{j \to \infty} G( z^{t_{j}} ) \leq G( z^{*} ).
\end{equation}
On the other hand, since $G$ is a lower semi-continuious, we have $\lim\inf_{j \to \infty} G( z^{t_{j}} ) \geq G(z^*).$ Hence
\begin{equation}\label{ineq20}
\lim_{j \to \infty} G( z^{t_{j}} ) = G( z^{*} ). 
\end{equation}
By summing \eqref{opti-1} and \eqref{opti-2} and taking limit along the convergent subsequence $\{(x^{t_j}, y^{t_j}, z^{t_j})\}$, and  applying \eqref{ineq20} and \eqref{subdiffproperty},  we have
\begin{equation}\label{ineq21}
0 \in \nabla F(y^*) + \partial G(y^*) + \nabla H(y^*).
\end{equation} 
This completes the proof. \hfill$\square$ \\

To prove \cref{whole-con}, we need the following lemma.
\begin{lemma}\label{subgradient}
Let \cref{ass1} be satisfied and $H$ be a twice continuously differentiable function with a bounded Hessian, i.e., there exists a constant $M > 0$ such that $\| \nabla H(y) \|_{2} \leq M$ for all $y$. Let $\{ (x^{t}, y^{t}, z^{t})\}_{t \geq 0}$ be a sequence generated by \cref{alg:DYS}. Then, there exists $\tau > 0$ such that for any $t \geq 1$,
\begin{equation}\label{ineq22}
\textmd{dist}  (0, \partial \Theta_{\gamma} (x^t, y^t, z^t)) \leq \tau \|y^{t+1} - y^t\|. 
\end{equation}
\end{lemma}
\begin{proof}
It is easy to compute that for any $t \geq 0$,  
\begin{equation}\label{ineq23}
\begin{aligned}
\nabla_x\Theta_\gamma \left(  x^{t+1}, y^{t+1}, z^{t+1} \right) &= \frac{1}{\gamma} \Big( z^{t+1} - y^{t+1} \Big) \\
& =\frac{1}{\gamma} \Big( x^{t+1} - x^{t} \Big),  
\end{aligned}
\end{equation}
where the last equality follows from \eqref{algx}. Secondly, we compute the subgradient of $\Theta_{\gamma}$ with respect to $z$, we get
\begin{equation}\label{eq400}
\begin{aligned}
& \nabla_z   \Theta_{\gamma} \Big( x^{t+1}, y^{t+1}, z^{t+1} \Big)\\
& = \partial G \Big( z^{t+1} \Big) + \frac{1}{\gamma} \Big( z^{t+1} - 2 y^{t+1} + \gamma \nabla H(y^{t+1}) + x^{t+1} \Big)  - \frac{2}{\gamma} \Big( z^{t+1} - y^{t+1} \Big) \\
& = \partial G \Big( z^{t+1} \Big) + \frac{1}{\gamma}\Big( z^{t+1} - 2 y^{t+1} + \gamma \nabla H(y^{t+1}) + x^t \Big) + \frac{1}{\gamma} \Big( x^{t+1} - x^{t} \Big)\\
& ~~~- \frac{2}{\gamma} \Big( z^{t+1} - y^{t+1} \Big) \\
& \ni -\frac{1}{\gamma} \Big(x^{t+1} - x^t \Big), 
\end{aligned} 
\end{equation}
where the second equality is achieved by adding $\frac{1}{\gamma} x^{t}$ and subtracting it at the same time and the inclusion follows from \eqref{opti-2} and \eqref{algx}. Finally, for the subgradient of $\Theta_{\gamma}$ with respect to $y$, we have 
\begin{equation}\label{ineq24}
\begin{aligned}
&\partial_y \Theta_\gamma \Big( x^{t+1}, y^{t+1}, z^{t+1} \Big) \\
& = \nabla F\Big( y^{t+1} \Big) + \frac{1}{\gamma} \Big(  y^{t+1} - x^{t+1} \Big) + \nabla^{2} H \Big( y^{t+1} \Big) \Big( z^{t+1} - y^{t+1} \Big) \\
& = \nabla F\Big(y^{t+1} \Big) + \frac{1}{\gamma} \Big( y^{t+1} - x^{t} \Big) + \frac{1}{\gamma} \Big( x^{t} - x^{t+1} \Big) + \nabla^{2} H \Big(y^{t+1} \Big) \Big( z^{t+1} - y^{t+1} \Big)\\
& = \frac{1}{\gamma} \Big( x^{t} - x^{t+1} \Big) + \nabla^{2} H \Big(y^{t+1} \Big) \Big( z^{t+1} - y^{t+1} \Big),\\
\end{aligned}
\end{equation}
where we have used the optimization condition \eqref{opti-1}. By the boundedness of the $\nabla^{2}H(y)$, we get
\begin{equation}\label{ineq25}
\begin{aligned}
&\| \partial_y \Theta_\gamma \Big( x^{t+1}, y^{t+1}, z^{t+1} \Big) \|  \\
& \leq \frac{1}{\gamma} \| x^t - x^{t+1} \| + M  \| x^{t+1} - x^{t} \| \\
&\leq \Big( \frac{1}{\gamma} + M \Big) \| x^{t} - x^{t+1} \|.
\end{aligned}
\end{equation}
It follows from \eqref{ineq23}, \eqref{eq400} and \eqref{ineq25} that there exists some constant $\tau >0$ such that whenever $t \geq 1$, we have 
\begin{equation}\label{ineq26}
\textmd{dist}  (0, \partial \Theta_\gamma (x^t, y^t, z^t)) \leq \tau \|y^{t+1} - y^t\|. 
\end{equation}
\end{proof}

{\it Proof of \cref{whole-con}.}
Firstly, we show that the statement $(i)$ holds. It follows from \eqref{de-ineq} that there exists $\Lambda(\gamma)>0 $ such that
\begin{equation}\label{ineq28}
\Theta_\gamma(x^t, y^t, z^t) - \Theta_\gamma(x^{t+1}, y^{t+1}, z^{t+1}) \geq \Lambda(\gamma) \|y^{t+1} - y^t\|^2. 
\end{equation}
Hence, $\Theta_\gamma(x^t, y^t, z^t)$ is nonincreasing. Let $\{(x^{t_i}, y^{t_i}, z^{t_i})\}$ be a convergent subsequence which converges to $(x^*, y^*, z^*)$.  Then, by the lower semicontinuity of $\Theta_\gamma$, we know that the sequence $\{ \Theta_\gamma(x^{t_i}, y^{t_i}, z^{t_i}) \}$ is bounded below. This together with the nonincreasing property of $\Theta_\gamma(x^t, y^t, z^t)$ implies that $\Theta_\gamma(x^t, y^t, z^t)$ is also bounded below. Therefore,  $\lim_{t \to \infty}\Theta_\gamma(x^t, y^t, z^t) = \Theta^*$ exists. We claim that $\Theta^*=\Theta_\gamma(x^*, y^*, z^*)$. Indeed, let $\{(x^{t_j}, y^{t_j}, z^{t_j})\}$ be any sequence that converges to $(x^*, y^*, z^*)$. Then by the lower semicontinuity, we have
\begin{equation}\label{ineq29}
\liminf_{j \to \infty} \Theta_\gamma(x^{t_j}, y^{t_j}, z^{t_j}) \geq \Theta_\gamma(x^*, y^*, z^*). 
\end{equation}
Moreover,  similar to \eqref{ineq17}, \eqref{ineq18} and \eqref{ineq19},  we also have 
\begin{equation}\label{ineq30}
\limsup_{j \to \infty} \Theta_\gamma(x^{t_j}, y^{t_j}, z^{t_j}) \leq \Theta_\gamma(x^*, y^*, z^*). 
\end{equation}
Now we easily get $\Theta^*=\Theta_\gamma(x^*, y^*, z^*)$, as claimed. 

In the next, we prove the second statement $(ii)$. We consider two cases.

 {\bf Case 1.} If $\Theta_\gamma(x^{t_0}, y^{t_0}, z^{t_0}) =\Theta^*$ for some $t_0 \geq 1$, then $\Theta_\gamma(x^{t_0+k}, y^{t_0+k}, z^{t_0+k}) = \Theta_\gamma(x^{t_0}, y^{t_0}, z^{t_0})$ for all $k\geq 0$ since the sequence is nonincreasing. Then from \eqref{ineq28}, we have $y^{t_0+k}=y^{t_0}$ for all $k\geq 0$. By \eqref{xineq}, we see that $x^{t_0+k}=x^{t_0}$ for all $k\geq 0$. These together with \eqref{algx}  show that we also have $z^{t_0 +k} = z^{t_0 }$ for all $k\geq 1$. Thus, the sequence $(x^t, y^t, z^t)$ remains constant  starting with the $(t_0+1)$st iteration.   Hence, the theorem holds trivially when this happens. 
 
 {\bf Case 2.} $\Theta_\gamma(x^{t}, y^{t}, z^{t})  > \Theta^*$ for any $t \geq 1$. 
We will show $\{\| y^{t+1} - y^t\|\}$ is summable. Recall that the function
$$
(x, y, z) \longmapsto \Theta_\gamma (x, y, z)
$$
is a KL function. By the property of KL function, there exist $\eta >0$, a neighborhood $U$ of $(x^*, y^*, z^*)$ and a continuous concave function $\varphi : [0, \eta) \to \mathbb{R}_+$ such that for all $(x, y, z) \in U$ satisfying $\Theta^* <  \Theta_\gamma(x, y, z) < \Theta^* + \eta$, we have 
\begin{equation}\label{ineq31}
\varphi^\prime ( \Theta_\gamma(x, y, z) - \Theta^*) \textmd{dist} (0, \partial \Theta_\gamma(x, y, z) ) \geq 1. 
\end{equation}
Since $U$ is an open set, take $\rho >0$ such that 
\begin{equation}\label{ineq33}
\mathbf{B}_\rho := \{ (x, y, z) : \|y - y^*\| < \rho, \|z - z^*\| < 2\rho, \|x - x^*\| < (2 + \gamma L) \rho\} \subseteq U 
\end{equation}
and set $B_\rho:= \{y: \|y - y^*\| <\rho \}$. From \cref{x-bounded}, we can get
\begin{equation}\label{ineq34}
\|x^t -x^*\| \leq \|x^t - x^{t-1}\| + \|x^{t-1} - x^*\| \leq \|x^t - x^{t-1}\| + (1 + \gamma L) \|y^t - y^*\|.  
\end{equation}
By \cref{clus-point}, there exists $N_0 \geq 1$ such that $\|x^t - x^{t-1}\| < \rho$ whenever $t \geq N_0$.  Hence, it follows that $\|x^t  - x^*\| < (2 + \gamma L)\rho$ whenever $y^t \in B_\rho$ and $t \geq N_0$.  
Applying \eqref{algx}, we also have that whenever $y^t \in B_\rho$ and for $t \geq N_0$, 
\begin{equation}\label{ineq35}
\|z^t - z^*\| \leq \|y^t - y^*\| + \|x^t - x^{t-1}\| < 2\rho.
\end{equation}
Thus, we obtain that if $y^t \in B_\rho$ and $t \geq N_0$, then $(x^t, y^t, z^t) \in \mathbf{B}_\rho \subseteq U$.  Now, by the facts that $(x^*, y^*, z^*) $ is a cluster point, that $\Theta_\gamma(x^t, y^t, z^t) > \Theta^*$ for every $t \geq 1$, and that $\lim_{t\to \infty}\Theta_\gamma(x^t, y^t, z^t) = \Theta^*$, there exists $(x^N, y^N, z^N)$ with $N \geq N_0$ such that
\begin{itemize}
\item[(i)] $y^N \in B_\rho$ and $\Theta^* < \Theta_\gamma(x^N, y^N, z^N) < \Theta^* +\eta$;

\item[(ii)] $\|y^N - y^*\| + \frac{\tau}{\Lambda(\gamma)} \varphi (\Theta_\gamma (x^N, y^N, z^N) - \Theta^*) < \rho$. 
\end{itemize}

Next, we prove that whenever $y^t \in B_\rho$ and $\Theta^* < \Theta_\gamma (x^t, y^t, z^t) < \Theta^* + \eta$ for some $t \geq N_0$, we have
\begin{equation}\label{ineq36}
\|y^{t+1} - y^t\| \leq \frac{\tau}{\Lambda(\gamma)} \Big[ \varphi \big( \Theta_{\gamma}(x^t, y^t, z^t) - \Theta^* \big)  - \varphi \big( \Theta_{\gamma}(x^{t+1}, y^{t+1}, z^{t+1}) - \Theta^* \big)  \Big].
\end{equation}
Recall that $\{ \Theta_{\gamma}(x^t, y^t, z^t) \}$ is non-increasing and $\varphi$ is increasing, \eqref{ineq36} holds obviously if $y^t=y^{t+1}$. Without loss generality, we assume that $y^{t+1} \neq y^t$. Since $y^t \in B_\rho$ and $t \geq N_0$, we have $(x^t, y^t, z^t) \in \mathbf{B}_\rho \subseteq U$. Hence, \eqref{ineq31} holds for $(x^t, y^t, z^t)$. Using \eqref{ineq26}, \eqref{ineq28}, \eqref{ineq31} and the concavity of $\varphi$,  we obtain  that for such $t$, 
\begin{equation}\label{ineq37}
\begin{aligned}
& \tau \|y^{t+1} - y^t\| \cdot \Big[\varphi \big( \Theta_\gamma(x^t, y^t, z^t) - \Theta^* \big) - \varphi \big( \Theta_\gamma(x^{t+1}, y^{t+1}, z^{t+1})  - \Theta^* \big)  \Big] \\
&~~ \geq \textmd{dist} (0, \partial \Theta_\gamma(x^t, y^t, z^t)) \cdot \Big[\varphi \big( \Theta_\gamma(x^t, y^t, z^t) - \Theta^* \big) - \varphi \big( \Theta_\gamma(x^{t+1}, y^{t+1}, z^{t+1})  - \Theta^* \big)  \Big] \\
&~~ \geq \textmd{dist} (0, \partial \Theta_\gamma(x^t, y^t, z^t)) \cdot \varphi^\prime \big(\Theta_\gamma(x^t, y^t, z^t) - \Theta^* \big) \\
&~~~~~~~\cdot \Big[ \Theta_\gamma(x^t, y^t, z^t)  -  \Theta_\gamma(x^{t+1}, y^{t+1}, z^{t+1}) \Big] \\
&~~ \geq \Lambda(\gamma) \|y^{t+1} - y^t\|^2 . 
\end{aligned}
\end{equation}
This implies that \eqref{ineq36} holds immediately. 

We next claim that $y^t \in B_\rho$ for all $t \geq N$.  First, the claim is true whenever $t=N$ by construction. Now, suppose that the claim is true for $t=N, \dots, N+k-1$ for some $k \geq 1$, that is, $y^N, \dots, y^{N+k-1} \in B_\rho$. Note that $\Theta^* < \Theta_\gamma(x^t, y^t, z^t) < \Theta^* + \eta$ for all $t \geq N$ by the choice of $N$ and non-increase property of $\{\Theta_\gamma(x^t, y^t, z^t)\}$. Hence, \eqref{ineq36} can be used for $t=N, \dots, N+k-1$. Thus, for $t=N+k$, we have 
\begin{equation}\label{ineq38}
\begin{aligned}
\| y^{N+k} - y^*\| & \leq \|y^N - y^* \| + \sum_{j=1}^k \|y^{N+j} - y^{N+j-1}\| \\
&  \leq \|y^N -y^*\| + \frac{\tau}{\Lambda(\gamma)}\sum_{j=1}^k \Big[ \varphi \big( \Theta_{\gamma}(x^{N+j-1}, y^{N+j-1}, z^{N+j-1}) - \Theta^* \big) \\
&~~~  - \varphi \big( \Theta_{\gamma}(x^{N+j}, y^{N+j}, z^{N+j}) -\Theta^* \big)  \Big]\\
& \leq \|y^N -y^*\| + \frac{\tau}{\Lambda(\gamma)} \varphi \big( \Theta_{\gamma}(x^{N}, y^{N}, z^{N}) - \Theta^* \big) < \rho. 
\end{aligned}
\end{equation}
Hence, $y^{N+k} \in B_\rho$. By induction, we obtain that $y^t \in B_\rho$ for all $t \geq N$. 

Note that  we have shown that  $y^t \in B_\rho$ and $\Theta^* < \Theta_\gamma(x^t, y^t, z^t) < \Theta^* + \eta$ for all $t \geq N$. Summing \eqref{ineq36} from $t=N$ to $M$ and letting $M \to \infty$, we obtain
\begin{equation}\label{ineq39}
\sum_{t=N}^\infty \|y^{t+1} - y^t\| \leq \frac{\tau}{\Lambda(\gamma)} \varphi \big( \Theta_{\gamma}(x^N, y^N, z^N) - \Theta^* \big)< +\infty.  
\end{equation}
This shows that $\{\| y^{t+1} - y^t\|\}$ is summable and hence the whole sequence $\{y^t\}$ converges  to $y^*$. From this and \cref{x-bounded} we obtain that $\{ \| x^{t+1} - x^{t} \| \}$ is summable and that the sequence $\{x^t\}$ is convergent. Finally, by  \eqref{algx}, we know that $\{ \| z^{t+1} - z^{t} \| \}$ is summable and the convergence of $\{z^t\}$ follows. The proof is completed. \hfill$\square$\\

{\it Proof of \cref{conv-rate}.}
Let $q_t = \Theta_{\gamma}(x^t, y^t, z^t) - \Theta_{\gamma}(x^*, y^*, z^*)$. Then, we have from the \cref{decrease} and \cref{whole-con} $(i)$ that $q_t \geq 0$ for all $t \geq 1$ and $q_t \to 0$ as $t \to \infty$. Furthermore, by \eqref{de-ineq}, we have 
\begin{equation}\label{eq50}
q_t - q_{t+1} \geq \Lambda(\gamma) \| y^{t+1} - y^t \|^2.
\end{equation}
Because of $q^{t+1} \geq 0$, it follows that $\Lambda(\gamma) \| y^{t+1} - y^{t} \|^2 \leq q_t - q_{t+1} \leq q_t$ for all $t \geq 1$. This together with \cref{x-bounded} implies that
\begin{equation}\label{eq51}
\| y^t - z^t \| = \| x^t - x^{t-1} \| \leq (1 + \gamma L) \| y^{t+1} - y^{t} \| \leq \frac{1 + \gamma L}{\sqrt{\Lambda(\gamma)}} \sqrt{q_t},
\end{equation}
where the first equality follows from \eqref{algx}. Adding \eqref{opti-1} and \eqref{opti-2} we have
\begin{equation}\label{eq52}
0 \in \nabla F(y^{t}) + \partial G(z^t) + \nabla H(y^t) + \frac{1}{\gamma}(z^t - y^t).
\end{equation}
This with the Lipschitz continuity of $\nabla F$ and $\nabla H$ yields that
\begin{equation}\label{eq53}
\textmd{dist}(0, \nabla F(z^t) + \partial G(z^t) + \nabla H(z^t)) \leq (L + \beta + \frac{1}{\gamma})\| z^t - y^t \|. 
\end{equation}
Therefore, for all $t \geq 1$,
\begin{equation}\label{eq54}
\textmd{dist}(0, \nabla F(z^t) + \partial G(z^t) + \nabla H(z^t)) \leq (L + \beta + \frac{1}{\gamma})\frac{1+\gamma L}{\sqrt{\Lambda(\gamma)}} \sqrt{q_t}.
\end{equation}
In the following, the estimation of $q_t$ is similar to the proof in many papers, such as see \cite{BDL, BST, LP}, so here we omit the rest of the proof.\hfill$\square$\\

%

\bibliographystyle{siamplain}
\bibliography{reference}


\end{document}